\documentclass{amsart}
\usepackage{amsfonts,amssymb, wasysym}
\usepackage{multicol}
\usepackage{hyperref}
\hypersetup{colorlinks,citecolor=blue,linkcolor=blue}

\newtheorem{lemma}{Lemma}

\newcommand{\integers}{{\mathbb Z}}
\newcommand{\realnos}{{\mathbb R}}

\def\ov{\overline}

\def\Nu{{\rm N}}
\def\Mu{{\rm M}}
\def\Tau{{\rm T}}
\def\Eta{{\rm H}}

\def\Kappa{{\rm K}}

\def\Rho{{\rm P}}

\usepackage{hyperref}
\hypersetup{colorlinks,citecolor=blue,linkcolor=blue}

\begin{document}

\title{Co-Seifert Fibrations of Compact Flat 3-Orbifolds}

\author{John G. Ratcliffe and Steven T. Tschantz}

\address{Department of Mathematics, Vanderbilt University, Nashville, TN 37240
\vspace{.1in}}

\email{j.g.ratcliffe@vanderbilt.edu}

\date{}

\begin{abstract}
This paper is a continuation of our previous paper \cite{R-T-Co} in which we developed 
the theory for classifying geometric fibrations 
of compact, connected, flat $n$-orbifolds, over a 1-orbifold, up to affine equivalence. 
In this paper, we apply our theory to classify all the geometric fibrations 
of compact, connected, flat $3$-orbifolds, over a 1-orbifold, up to affine equivalence. 
\end{abstract}

\maketitle

\section{Introduction}\label{S:1} 
An {\it $n$-dimensional crystallographic group} ({\it $n$-space group}) 
is a discrete group $\Gamma$ of isometries of Euclidean $n$-space $E^n$ 
whose orbit space $E^n/\Gamma$ is compact. 
If $\Gamma$ is an $n$-space group, then $E^n/\Gamma$ is a compact, connected, 
flat $n$-orbifold, and conversely if $M$ is a compact, connected, flat $n$-orbifold, 
then there is an $n$-space group $\Gamma$ such that $M$ is isometric to $E^n/\Gamma$. 

In our paper \cite{R-T}, we proved that a geometric fibration of $E^n/\Gamma$ 
corresponds to a space group extension 
$$1 \to \Nu \hookrightarrow \Gamma \to \Gamma/\Nu\to 1.$$
The corresponding geometric fibration of $E^n/\Gamma$ is said to be a {\it co-Seifert fibration} when the base space is a 1-orbifold, or equivalently, when $\Gamma/\Nu$ is a 1-space group. 

In our previous paper \cite{R-T-Co}, we develop the theory for classifying all co-Seifert geometric fibrations of compact, connected, flat $n$-orbifolds up to affine equivalence. 
By Theorems 5 and 10 of \cite{R-T}, this problem is equivalent to classifying 
all pairs $(\Gamma,\Nu)$ such that $\Gamma$ is an $n$-space 
group and $\Nu$ is a normal subgroup of $\Gamma$, such that $\Gamma/\Nu$ is 
infinite cyclic or infinite dihedral, up to isomorphism. 
In our paper \cite{R-T-B}, we proved that for each dimension $n$, there are only finitely many isomorphism classes of such pairs. In \cite{R-T-Co}, we described the classification for $n = 2$.  
In this paper, we describe the classification for $n = 3$. 
In the process, we determine the group of affinities of every compact, connected, flat $2$-orbifold. 
In particular, we show that the group of affinities of a flat torus has an interesting free product 
with amalgamation decomposition. 
Finally, as an application, we give an explanation for all but one of the enantiomorphic 3-space group pairs.

\section{Organization and Description of the Classification} 

This paper is a continuation of \cite{R-T-Co}, and we will refer to \cite{R-T-Co} for 
all definitions and basic results.  As explained in \S 9 of \cite{R-T-Co}, the classification 
follows from knowledge of the action of the structure group. 
In this paper, we describe the action of the structure group for the generalized Calabi constructions 
corresponding to the Seifert and dual co-Seifert geometric fibrations of compact, connected, 
flat 3-orbifolds given in Table 1 of \cite{R-T}. 
We did not describe the co-Seifert fibration projections in Table 1 of \cite{R-T} 
because we did not have an efficient way of giving a description. 
Our generalized Calabi construction gives us a simple way to simultaneously describe 
both a Seifert fibration and its dual co-Seifert fibration of a compact, connected, flat 3-orbifold 
via Theorem 4 and Corollary 1 of \cite{R-T-Co} (see also \cite{R-T-C}).

We will organize our description into 17 tables, 
one for each 2-space group type of the generic fiber of the co-Seifert fibration. 
We will go through the 2-space groups in reverse order because 
the complexity of the description generally increases in reverse order of the IT number of a 2-space group. 
That the computer generated affine classification of the co-Seifert geometric fibrations 
described in Table 1 of \cite{R-T} is correct and complete follows from Theorems 7 -- 12 of \cite{R-T-Co} 
and Lemmas \ref{L1} -- \ref{L33} below. 

\begin{enumerate}
\item In each of these 17 tables, the first column lists the IT (international tables) number of the 3-space group $\Gamma$. 

\item The second column lists the generic fibers $(V/\Nu, V^\perp/\Kappa$) of the co-Seifert and Seifert fibrations 
corresponding to a 2-dimensional, complete, normal subgroup $\Nu$ of $\Gamma$ with $V = \mathrm{Span}(\Nu)$ and $\Kappa = \Nu^\perp$. 
We will denote a circle by $\mathrm{O}$ and a closed interval by $\mathrm{I}$. 

\item The third column lists the isomorphism type of the structure group $\Gamma/\Nu\Kappa$,  
with $C_n$ indicating a cyclic group of order $n$,  
and $D_n$ a dihedral group of order $2n$. 

\item The fourth column lists the quotients $(V/(\Gamma/\Kappa),V^\perp/(\Gamma/\Nu))$ 
under the action of the structure group.  
Note that $(V/(\Gamma/\Kappa), V^\perp/(\Gamma/\Nu))$  are the bases of the Seifert and co-Seifert fibrations. 

\item The fifth column indicates how generators of the 
structure group act diagonally on the Cartesian product $V/\Nu \times V^\perp/\Kappa$. 
The structure group action was derived from the standard affine representation 
of $\Gamma$ in Table 1B of \cite{B-Z}. 
We denote a rotation of $\mathrm{O}$ of $(360/n)^\circ$ by $n$-rot., 
and a reflection of $\mathrm{O}$ or $\mathrm{I}$ by ref. 
We denote the identity map by idt. 

\item The sixth column lists the classifying pairs for the co-Seifert fibrations. 
These are pairs $\{\alpha, \beta\}$ of affinities of $V/\Nu$ that classify co-Seifert fibrations 
via Theorems 7 -- 10 of \cite{R-T-Co}. 
For $C_n$ actions, $\alpha$ and $\beta$ are inverse affinities of order $n$. 
For $D_1$ actions, $\alpha = \mathrm{idt.}$, and $\beta$ has order 2. 
For $D_n$ actions,  with $n > 1$, $\alpha$ and $\beta$ are affinities of order 2 
whose product has order $n$.  In particular, for $D_n$ actions given by 
($\alpha$, ref.), ($\gamma$, $n$-rot.), the classifying pair is $\{\alpha,\beta\}$ 
with $\gamma = \alpha\beta$, except for the cases $n = 3, 6$ at the end of Tables \ref{T16} and \ref{T17},  
which require an affine deformation.  
We indicate a classifying pair that falls into the case $E_1\cap E_2 \neq \{0\}$ 
of Theorem 10 of \cite{R-T-Co} by an asterisk. 
\end{enumerate}

Tables \ref{T1} -- \ref{T17} were first computed by hand, and then they were double checked by a computer calculation.

Let $C_\infty$ be the standard infinite cyclic 1-space group $\langle e_1+I\rangle$, 
and let $D_\infty$ be the standard infinite dihedral 1-space group $\langle e_1+I, -I\rangle$. 

If $\Mu$ is a 2-space group, 
we define the {\it symmetry group} of the flat orbifold $E^2/\Mu$ 
by $\mathrm{Sym}(\Mu) = \mathrm{Isom}(E^2/\Mu)$. 
We will identify $V/\Nu$ with $E^2/\mathrm{M}$ and the group of affinities $\mathrm{Aff}(V/\Nu)$ of $V/\Nu$ 
with $\mathrm{Aff(M)} =\mathrm{Aff}(E^2/\mathrm{M})$.

\section{Generic Fiber $\ast632$ ($30^\circ - 60^\circ$ right triangle) with IT number 17} 

The 2-space group $\Mu$ with IT number 17 is $\ast632$ in Conway's notation \cite{Conway} or $p6m$ in IT notation.  
See space group 4/4/1/1 in Table 1A of \cite{B-Z} for the standard affine representation of $\Mu$. 
The flat orbifold $E^2/\Mu$ is a $30^\circ-60^\circ$ right triangle. 

\begin{lemma}\label{L1} 
If $\Mu$ is the 2-space group $\ast 632$, 
then $\mathrm{Sym}(\Mu) = \mathrm{Aff}(\Mu) =\{\rm idt.\}$, 
and $\Omega:  \mathrm{Aff}(\Mu) \to \mathrm{Out}(\Mu)$ is an isomorphism. 
\end{lemma}
\begin{proof} $\mathrm{Sym}(\Mu) =  \{\rm idt.\}$, 
since a symmetry of $E^2/\Mu$ fixes each corner point. 
We have that $Z(\Mu) = \{I\}$ by Lemma 5 of \cite{R-T-I}. 
Hence $\Omega: \mathrm{Sym}(\Mu) \to \mathrm{Out}_E(\Mu)$ is an isomorphism 
by Theorems 1 and 2 of \cite{R-T-I}. 
We have that $\mathrm{Out}_E(\Mu) = \mathrm{Out}(\Mu)$  by Lemma 9 of \cite{R-T-I} and Table 5A of \cite{B-Z},  
and $\Omega: \mathrm{Aff}(\Mu) \to \mathrm{Out}(\Mu)$ is an isomorphism 
by Theorems 1 and 3 in \cite{R-T-I}. 
Therefore $\mathrm{Sym}(\Mu) = \mathrm{Aff}(\Mu)$. 
\end{proof}

We represent $\mathrm{Out}(\Mu)$ by $\mathrm{Sym}(\Mu)$ via the isomorphism 
$\Omega:  \mathrm{Sym}(\Mu) \to \mathrm{Out}(\Mu)$. 
The set $\mathrm{Isom}(C_\infty, \Mu)$ consists of one element corresponding 
to the pair of inverse elements \{idt., idt.\} of $\mathrm{Sym}(\Mu)$ 
by Theorem 7 of \cite{R-T-Co}. 
The corresponding affine equivalence class of co-Seifert fibrations corresponds  
to row 1 of Table \ref{T1}. 

The set $\mathrm{Isom}(D_\infty, \Mu)$ consists of one element corresponding 
to the pair of indentity elements \{idt., idt.\} of $\mathrm{Sym}(\Mu)$ 
by Theorem 8 of \cite{R-T-Co}. 
The corresponding affine equivalence class of co-Seifert fibrations corresponds  
to row 2 of Table \ref{T1}.

Notice that the first row of Table \ref{T1} says that $E^3/\Gamma$ is the Cartesian product $\ast632 \times \mathrm{O}$, and the second row of Table \ref{T1} says that $E^3/\Gamma$ is the Cartesian product $\ast632 \times \mathrm{I}$. The corresponding co-Seifert geometric fibrations of $E^3/\Gamma$ are the Cartesian product fibrations of $\ast632 \times \mathrm{O}$ and $\ast632 \times \mathrm{I}$ 
with fiber $\ast632$. 

\begin{table}  
\begin{tabular}{llllll}
no. & fibers & grp. & quotients &  group action & classifying pair \\
\hline 
183  & $(\ast632, \mathrm{O})$ & $C_1$ & $(\ast632, \mathrm{O})$ & (idt., idt.)  & \{idt., idt.\} \\
191  & $(\ast632, \mathrm{I})$   & $C_1$ & $(\ast632, \mathrm{I})$ & (idt., idt.) & \{idt., idt.\} \end{tabular}

\medskip
\caption{The classification of the co-Seifert fibrations of 3-space groups 
whose generic fiber is of type $\ast 632$ with IT number 17}\label{T1}
\end{table}

\section{Generic Fiber $632$ (turnover) with IT number 16}  

The 2-space group $\Mu$ with IT number 16 is $632$ in Conway's notation or $p6$ in IT notation. 
See space group 4/3/1/1 in Table 1A of \cite{B-Z} for the standard affine representation of $\Mu$. 
The flat orbifold $E^2/\Mu$  is a turnover with 3 cone points 
obtained by gluing together two congruent $30^\circ-60^\circ$ right triangles along their boundaries. 
The $632$ turnover is orientable. 
Let c-ref.\ denote the {\it central reflection} between the two triangles. 

\begin{lemma}\label{L2} 
If $\Mu$ is the 2-space group $632$, 
then $\mathrm{Sym}(\Mu) = \mathrm{Aff}(\Mu) =$ \{\rm idt., c-ref.\}, 
and $\Omega:  \mathrm{Aff}(\Mu) \to \mathrm{Out}(\Mu)$ is an isomorphism. 
\end{lemma}
\begin{proof} $\mathrm{Sym}(\Mu)$ = \{\rm idt., c-ref.\}, 
since a symmetry of $E^2/\Mu$ fixes each cone point. 
We have that $\mathrm{Sym}(\Mu) = \mathrm{Aff}(\Mu)$ 
and $\Omega:  \mathrm{Aff}(\Mu) \to \mathrm{Out}(\Mu)$ is an isomorphism as in Lemma \ref{L1}. 
\end{proof}

\begin{table} 
\begin{tabular}{llllll}
no. & fibers & grp. & quotients &  group action & classifying pair \\
\hline 
168 &  $(632, \mathrm{O})$ & $C_1$ & $(632, \mathrm{O})$ & (idt., idt.) & \{idt., idt.\}  \\
175 &  $(632, \mathrm{I})$    & $C_1$ & $(632, \mathrm{I})$ & (idt., idt.) & \{idt., idt.\} \\
177 &  $(632, \mathrm{O})$ & $C_2$ & $(\ast632, \mathrm{I})$ & (c-ref., ref.)  & \{c-ref., c-ref.\} \\
184 &  $(632, \mathrm{O})$    & $C_2$ & $(\ast632, \mathrm{O})$ & (c-ref., 2-rot.) & \{c-ref., c-ref.\} \\
192 &  $(632, \mathrm{I})$    & $D_1$ & $(\ast632, \mathrm{I})$ & (c-ref., ref.) & \{idt., c-ref.\} 
\end{tabular}

\medskip
\caption{The classification of the co-Seifert fibrations of 3-space groups 
whose generic fiber is of type $632$ with IT number 16}\label{T2}
\end{table}

We represent $\mathrm{Out}(\Mu)$ by $\mathrm{Sym}(\Mu)$ via the isomorphism 
$\Omega:  \mathrm{Sym}(\Mu) \to \mathrm{Out}(\Mu)$. 
The set $\mathrm{Isom}(C_\infty, \Mu)$ consists of two elements corresponding 
to the pairs of inverse elements \{idt., idt.\} and \{c-ref., c-ref.\} of $\mathrm{Sym}(\Mu)$ 
by Theorem 7 of \cite{R-T-Co}. 
The corresponding affine equivalence classes of co-Seifert fibrations correspond 
to the two rows of Table \ref{T2} whose column 4 second quotient is $\mathrm{O}$. 

The set $\mathrm{Isom}(D_\infty, \Mu)$ consists of three elements corresponding 
to the pairs of elements \{idt., idt.\}, \{c-ref., c-ref.\}, \{idt., c-ref.\} of order 1 or 2 of $\mathrm{Sym}(\Mu)$ 
by Theorem 8 of \cite{R-T-Co}. 
The corresponding affine equivalence classes of co-Seifert fibrations correspond  
to the three rows of Table \ref{T2} whose column 4 second quotient is $\mathrm{I}$.

\section{Generic Fiber $3{\ast}3$ (cone) with IT number 15} 

The 2-space group $\Mu$ with IT number 15 is $3{\ast}3$ in Conway's notation or $p31m$ in IT notation. 
See space group 4/2/2/1 in Table 1A of \cite{B-Z} for the standard affine representation of $\Mu$. 
The  flat orbifold $E^2/\Mu$ is a cone with one $120^\circ$ cone point and one  $60^\circ$ corner point   
obtained by gluing together two congruent $30^\circ-60^\circ$ right triangles along their sides 
opposite the $30^\circ$ and $90^\circ$ angles. 
Let c-ref.\ denote the {\it central reflection} between the two triangles. 

\begin{lemma}\label{L3} 
If $\Mu$ is the 2-space group $3\ast 3$, 
then $\mathrm{Sym}(\Mu) = \mathrm{Aff}(\Mu) =$ \{\rm idt., c-ref.\}, 
and $\Omega:  \mathrm{Aff}(\Mu) \to \mathrm{Out}(\Mu)$ is an isomorphism. 
\end{lemma}
\begin{proof} $\mathrm{Sym}(\Mu)$ = \{\rm idt., c-ref.\}, 
since a symmetry of $E^2/\Mu$ fixes the cone point and the corner point. 
We have that $\mathrm{Sym}(\Mu) = \mathrm{Aff}(\Mu)$ and $\Omega:  \mathrm{Aff}(\Mu) \to \mathrm{Out}(\Mu)$ is an isomorphism as in Lemma \ref{L1}. 
\end{proof}

The set $\mathrm{Isom}(C_\infty, \Mu)$ consists of two elements corresponding 
to the pairs of inverse elements \{idt., idt.\} and \{c-ref., c-ref.\} of $\mathrm{Sym}(\Mu)$ 
by Theorem 7 of \cite{R-T-Co}. 
The corresponding affine equivalence classes of co-Seifert fibrations correspond 
to the two rows of Table \ref{T3} whose column 4 second quotient is $\mathrm{O}$.

The set $\mathrm{Isom}(D_\infty, \Mu)$ consists of three elements corresponding 
to the pairs of elements \{idt., idt.\}, \{c-ref., c-ref.\}, \{idt., c-ref.\} of order 1 or 2 of $\mathrm{Sym}(\Mu)$ 
by Theorem 8 of \cite{R-T-Co}. 
The corresponding affine equivalence classes of co-Seifert fibrations correspond  
to the three rows of Table \ref{T3} whose column 4 second quotient is $\mathrm{I}$.

\begin{table} 
\begin{tabular}{llllll}
no. & fibers & grp. & quotients &  group action & classifying pair \\
\hline 
157 & $(3\!\ast\! 3, \mathrm{O})$ & $C_1$ & $(3\!\ast\! 3, \mathrm{O})$ & (idt., idt.) & \{idt., idt.\}  \\
162 & $(3\!\ast\! 3, \mathrm{O})$ & $C_2$ & $(\ast 632, \mathrm{I})$ & (c-ref., ref.) & \{c-ref., c-ref.\} \\
185 & $(3\!\ast\! 3, \mathrm{O})$ & $C_2$ & $(\ast 632, \mathrm{O})$ & (c-ref., 2-rot.) & \{c-ref., c-ref.\} \\
189 & $(3\!\ast\! 3, \mathrm{I})$   & $C_1$ & $(3\!\ast\! 3, \mathrm{I})$ & (idt., idt.) & \{idt., idt.\} \\
193 & $(3\!\ast\! 3, \mathrm{I})$   & $D_1$ & $(\ast 632, \mathrm{I})$ & (c-ref., ref.) & \{idt., c-ref.\}
\end{tabular}

\medskip
\caption{The classification of the co-Seifert fibrations of 3-space groups 
whose generic fiber is of type $3\!\ast\! 3$ with IT number 15}\label{T3}
\end{table}

\section{Generic Fiber $\ast 333$ (equilateral triangle) with IT number 14} 

The 2-space group $\Mu$ with IT number 14 is $\ast 333$ in Conway's notation or $p3m1$ in IT notation. 
See space group 4/2/1/1 in Table 1A of \cite{B-Z} for the standard affine representation of $\Mu$. 
The flat orbifold $E^2/\Mu$ is an equilateral triangle $\triangle$. 
The symmetry group of $\triangle$ is a dihedral group of order 6. 
There is one conjugacy class of symmetries of order 2 represented by a {\it triangle reflection}  t-ref.\ 
in the perpendicular bisector of a side. 
There is one conjugacy class of symmetries 
of order 3 represented by a rotation 3-rot.\ of $120^\circ$ about the center of the triangle. 
Define the triangle reflection t-ref.$'$ = (t-ref.)(3-rot.).  

\begin{lemma}\label{L4} 
If $\Mu$ is the 2-space group $\ast 333$, 
then $\mathrm{Sym}(\Mu)$ is the dihedral group  $\langle${\rm t-ref., 3-rot.}$\rangle$ of order 6, 
and $\mathrm{Sym}(\Mu) = \mathrm{Aff}(\Mu)$, 
and $\Omega:  \mathrm{Aff}(\Mu) \to \mathrm{Out}(\Mu)$ is an isomorphism. 
\end{lemma}
\begin{proof}  $\mathrm{Sym}(\Mu)$ = $\langle${\rm t-ref., 3-rot.}$\rangle$, 
since a symmetry permutes the corner points. 
We have that $\mathrm{Sym}(\Mu) = \mathrm{Aff}(\Mu)$ 
and $\Omega:  \mathrm{Aff}(\Mu) \to \mathrm{Out}(\Mu)$ is an isomorphism as in Lemma \ref{L1}. 
\end{proof}

We represent $\mathrm{Out}(\Mu)$ by $\mathrm{Sym}(\Mu)$ via the isomorphism 
$\Omega:  \mathrm{Sym}(\Mu) \to \mathrm{Out}(\Mu)$. 
The set $\mathrm{Isom}(C_\infty, \Mu)$ consists of three elements corresponding 
to the pairs of inverse elements \{idt., idt.\}, \{3-rot., 3-rot.$^{-1}$\}, \{t-ref., t-ref.\} of $\mathrm{Sym}(\Mu)$ 
by Theorem 7 of \cite{R-T-Co}. 
The corresponding affine equivalence classes of co-Seifert fibrations correspond 
to the three rows of Table \ref{T4} whose column 4 second quotient is $\mathrm{O}$. 

The set $\mathrm{Isom}(D_\infty, \Mu)$ consists of four elements corresponding 
to the pairs of elements \{t-ref., t-ref.\}, \{t-ref., t-ref.$'$\}, \{idt., idt.\}, \{idt., t-ref.\} of order 1 or 2 
of $\mathrm{Sym}(\Mu)$ by Theorem 8 of \cite{R-T-Co}. 
The corresponding affine equivalence classes of co-Seifert fibrations correspond 
to the four rows of Table \ref{T4} whose column 4 second quotient is $\mathrm{I}$.

\begin{table}  
\begin{tabular}{llllll}
no. & fibers & grp. & quotients &  structure group action & classifying pair \\
\hline 
156 & $(\ast 333, \mathrm{O})$ & $C_1$ & $(\ast 333, \mathrm{O})$ & (idt., idt.)  & \{idt., idt.\} \\
160 & $(\ast 333, \mathrm{O})$ & $C_3$ & $(3{\ast}3, \mathrm{O})$ & (3-rot., 3-rot.) & \{3-rot., 3-rot.$^{-1}$\} \\
164 & $(\ast 333, \mathrm{O})$ & $C_2$ & $(\ast 632, \mathrm{I})$ & (t-ref., ref.) & \{t-ref., t-ref.\} \\
166 & $(\ast 333, \mathrm{O})$ & $D_3$ & $(\ast 632, \mathrm{I})$ & (t-ref., ref.), (3-rot., 3-rot.) & \{t-ref., t-ref.$'$\}  \\
186 & $(\ast 333, \mathrm{O})$ & $C_2$ & $(\ast 632, \mathrm{O})$ & (t-ref., 2-rot.) & \{t-ref., t-ref.\} \\
187 & $(\ast 333, \mathrm{I})$   & $C_1$ & $(\ast 333, \mathrm{I})$ & (idt., idt.) & \{idt., idt.\} \\
194 & $(\ast 333, \mathrm{I})$   & $D_1$ & $(\ast 632, \mathrm{I})$ & (t-ref., ref.) & \{idt., t-ref.\} 
\end{tabular}

\medskip
\caption{The classification of the co-Seifert fibrations of 3-space groups 
whose generic fiber is of type $\ast 333$ with IT number 14}\label{T4}
\end{table}

\section{Generic Fiber $333$ (turnover) with IT number 13} 

The 2-space group $\Mu$ with IT number 13 is $333$ in Conway's notation or $p3$ in IT notation. 
See space group 4/1/1/1 in Table 1A of \cite{B-Z} for the standard affine representation of $\Mu$. 
The flat orbifold $E^2/\Mu$ is a turnover with three $120^\circ$ cone points 
obtained by gluing together two congruent equilateral triangles along their boundaries. 
The $333$ turnover is orientable. 
The symmetry group of this orbifold is the direct product of the subgroup of order 2, generated 
by the {\it central reflection} c-ref.\ between the two triangles, and the subgroup of order 6 corresponding 
to the symmetry group of the two triangles. 

There are 3 conjugacy classes of symmetries of order 2,  
the class of the {\it central reflection} c-ref., 
the class of the {\it triangle reflection} t-ref., 
and the class of the halfturn around a cone point 2-rot., defined so that 2-rot.\ = (c-ref.)(t-ref.). 
There is one conjugacy class of symmetries of order 3 represented by a rotation 3-rot.\ 
that cyclically permutes the cone points. 
There is one conjugacy class of dihedral subgroups of order 4, represented by the group 
\{idt., c-ref., t-ref., 2-rot.\}.  
There is one conjugacy class of symmetries of order 6 represented by 6-sym.\ = (c-ref.)(3-rot.).  
There are two conjugacy classes of dihedral subgroups of order 6, 
the class of the symmetry group of a triangular side of the turnover generated by t-ref.\ and 3-rot., 
and the class of the orientation-preserving subgroup generated by 2-rot.\ and 3-rot. 
Define t-ref.$'$ = (t-ref.)(3-rot.) and 2-rot.$'$ = (c-ref.)(t-ref.$'$). 

\begin{lemma}\label{L5} 
If $\Mu$ is the 2-space group $333$, 
then $\mathrm{Sym}(\Mu)$ is the dihedral group  $\langle${\rm c-ref., t-ref., 3-rot.}$\rangle$ of order 12, 
and $\mathrm{Sym}(\Mu) = \mathrm{Aff}(\Mu)$, 
and $\Omega:  \mathrm{Aff}(\Mu) \to \mathrm{Out}(\Mu)$ is an isomorphism.  
\end{lemma}
\begin{proof} $\mathrm{Sym}(\Mu) = \langle${c-ref., \rm t-ref., 3-rot.}$\rangle$, 
since a symmetry permutes the cone points.  
We have that $\mathrm{Sym}(\Mu) = \mathrm{Aff}(\Mu)$ and $\Omega:  \mathrm{Aff}(\Mu) \to \mathrm{Out}(\Mu)$ is an isomorphism as in Lemma \ref{L1}. 
\end{proof}

\begin{table}  
\begin{tabular}{llllll}
no. & fibers & grp. & quotients &  structure group action & classifying pair \\
\hline 
143 & $(333, \mathrm{O})$ & $C_1$ & $(333, \mathrm{O})$ & (idt., idt.)  & \{idt., idt.\} \\
146 & $(333, \mathrm{O})$ & $C_3$ & $(333, \mathrm{O})$ & (3-rot., 3-rot.) &  \{3-rot., 3-rot.$^{-1}$\} \\
147 & $(333, \mathrm{O})$ & $C_2$ & $(632, \mathrm{I})$ & (2-rot., ref.) & \{2-rot., 2-rot.\} \\
148 & $(333, \mathrm{O})$ & $D_3$ & $(632, \mathrm{I})$ & (2-rot., ref.), (3-rot., 3-rot.) &  \{2-rot., 2-rot.$'\}$ \\
149 & $(333, \mathrm{O})$ & $C_2$ & $(\ast 333, \mathrm{I})$ & (c-ref., ref.) & \{c-ref., c-ref.\} \\
150 & $(333, \mathrm{O})$ & $C_2$ & $(3{\ast}3, \mathrm{I})$ & (t-ref., ref.)  & \{t-ref., t-ref.\} \\
155 & $(333, \mathrm{O})$ & $D_3$ & $(\ast 333, \mathrm{I})$ & (t-ref., ref.), (3-rot., 3-rot.) & \{t-ref., t-ref.$'$\} \\
158 & $(333, \mathrm{O})$ & $C_2$ & $(\ast 333, \mathrm{O})$ & (c-ref., 2-rot.) &  \{c-ref., c-ref.\} \\
159 & $(333, \mathrm{O})$ & $C_2$ & $(3{\ast}3, \mathrm{O})$ & (t-ref., 2-rot.) &  \{t-ref., t-ref.\} \\
161 & $(333, \mathrm{O})$ & $C_6$ & $(3{\ast}3, \mathrm{O})$ & (6-sym., 6-rot.) & \{6-sym., 6-sym.$^{-1}$\} \\
163 & $(333, \mathrm{O})$ & $D_2$ & $(\ast 632, \mathrm{I})$ & (c-ref., ref.), (t-ref., 2-rot.) & \{c-ref., 2-rot.\}\\
165 & $(333, \mathrm{O})$ & $D_2$ & $(\ast 632, \mathrm{I})$ & (t-ref., ref.), (c-ref., 2-rot.) & \{t-ref., 2-rot.\} \\
167 & $(333, \mathrm{O})$ & $D_6$ & $(\ast 632, \mathrm{I})$ & (t-ref., ref.), (6-sym., 6-rot.) &  \{t-ref., 2-rot.$'$\} \\
173 & $(333, \mathrm{O})$ & $C_2$ & $(632, \mathrm{O})$ & (2-rot., 2-rot.)  & \{2-rot., 2-rot.\} \\
174 & $(333, \mathrm{I})$  & $C_1$ & $(333, \mathrm{I})$ & (idt., idt.)  & \{idt., idt.\} \\
176 & $(333, \mathrm{I})$  & $D_1$ & $(632, \mathrm{I})$ & (2-rot., ref.) &  \{idt., 2-rot.\} \\
182 & $(333, \mathrm{O})$ & $D_2$ & $(\ast 632, \mathrm{I})$ &  (c-ref., ref.), (2-rot., 2-rot.) & \{c-ref., t-ref.\} \\
188 & $(333, \mathrm{I})$   & $D_1$ & $(\ast 333, \mathrm{I})$ & (c-ref., ref.) &  \{idt., c-ref.\} \\
190 & $(333, \mathrm{I})$   & $D_1$ & $(3{\ast}3, \mathrm{I})$ & (t-ref., ref.)  & \{idt., t-ref.\} 
\end{tabular}

\medskip
\caption{The classification of the co-Seifert fibrations of 3-space groups 
whose generic fiber is of type $333$ with IT number 13}\label{T5}
\end{table}

We represent $\mathrm{Out}(\Mu)$ by $\mathrm{Sym}(\Mu)$ via the isomorphism 
$\Omega:  \mathrm{Sym}(\Mu) \to \mathrm{Out}(\Mu)$. 
The set $\mathrm{Isom}(C_\infty, \Mu)$ consists of six elements corresponding 
to the pairs of inverse elements \{idt., idt.\}, \{3-rot., 3-rot.$^{-1}$\}, \{c-ref., c-ref.\}, \{t-ref., t-ref.\}, 
\{6-sym., 6-sym.$^{-1}$\}, \{2-rot., 2-rot.\} of $\mathrm{Sym}(\Mu)$ 
by Theorem 7 of \cite{R-T-Co}.

The set $\mathrm{Isom}(D_\infty, \Mu)$ consists of thirteen elements corresponding 
to the pairs of elements \{2-rot., 2-rot.\}, \{2-rot., 2-rot.$'$\}, \{c-ref., c-ref.\}, \{t-ref., t-ref.\}, \{t-ref., t-ref.$'$\}, 
\{c-ref., 2-rot.\}, \{t-ref., 2-rot.\}, \{t-ref., 2-rot.$'$\}, \{idt., idt.\}, \{idt., 2-rot.\}, \{c-ref., t-ref.\}, \{idt., c-ref.\}, \{idt., t-ref.\} 
of order 1 or 2 of $\mathrm{Sym}(\Mu)$ by Theorem 8 of \cite{R-T-Co}. 

\medskip
\noindent{\bf Example 1.}
Let $\Gamma$ be the affine 3-space group with IT number 163 in Table 1B of \cite{B-Z}. 
Then $\Gamma = \langle t_1,t_2,t_3, A,\beta,C\rangle$ 
where $t_i = e_i+I$ for $i=1,2,3$ are the standard translations, 
and $\beta=\frac{1}{2}e_3+B$, and 
$$A = \left(\begin{array}{rrr} 0 & -1 & 0\\ 1 & -1 & 0 \\ 0 & 0 & 1  \end{array}\right),\ \ 
B = \left(\begin{array}{rrr} 0 & -1 & 0  \\ -1 & 0 & 0   \\ 0 & 0 & -1 \end{array}\right), \ \ 
C = \left(\begin{array}{rrr} -1 & 0 & 0  \\ 0 & -1 & 0   \\ 0 & 0 & -1 \end{array}\right).$$
The group $\Nu = \langle t_1,t_2, A\rangle$ is a complete normal subgroup 
of $\Gamma$ with $V = {\rm Span}(\Nu) = {\rm Span}\{e_1, e_2\}$. 
The flat orbifold $V/\Nu$ is a $333$ turnover.  
Let $\Kappa = \Nu^\perp$.  Then $\Kappa = \langle t_3\rangle$. 
The flat orbifold $V^\perp/\Kappa$ is a circle.  
The structure group $\Gamma/\Nu\Kappa$ is a dihedral group of order 4 
generated by $\Nu\Kappa\beta$ and $\Nu\Kappa C$. 
The elements $\Nu\Kappa\beta$ and $\Nu\Kappa C$ act on the circle $V^\perp/\Kappa$ as reflections. 
The elements  $\Nu\Kappa\beta$ and $\Nu\Kappa C$ both fix the cone point 
of $V/\Nu$ represented by $(0,0,0)$.  
The other two cone points of $V/\Nu$ are represented by 
$(2/3,1/3,0)$, which is the fixed point of $t_1A$, and by $(1/3,2/3,0)$,  
which is the fixed point of $t_1t_2A$. 
The element $\Nu\Kappa\beta$ acts as the central reflection of $V/\Nu$, 
since it fixes all three cone points.  
The element $\Nu\Kappa C$ acts as the halfturn around the cone point represented by $(0,0,0)$, 
since $\Nu\Kappa C$ preserves the orientation of $V/\Nu$ 
because $C$ preserves the orientation of $V$. 
By Theorems 8 and 12 of \cite{R-T-Co}, the classifying pair for the co-Seifert fibration determined by 
$(\Gamma, \Nu)$ is \{c-ref., 2-rot.\}. 

\section{Generic Fiber $4{\ast}2$ (cone) with IT number 12} 

The 2-space group $\Mu$ with IT number 12 is $4{\ast}2$ in Conway's notation or $p4g$ in IT notation. 
See space group 3/2/1/2 in Table 1A of \cite{B-Z} for the standard affine representation of $\Mu$. 
The flat orbifold $E^2/\Mu$ is a cone with one $90^\circ$ cone point and one $90^\circ$ corner point 
obtained by glueing together two congruent $45^\circ-45^\circ$ right triangles along two sides opposite 
$45^\circ$ and $90^\circ$ angles.  
Let c-ref.\ denote the  {\it central reflection} between the triangles. 

\begin{lemma}\label{L6} 
If $\Mu$ is the 2-space group $4\ast 2$, 
then $\mathrm{Sym}(\Mu) = \mathrm{Aff}(\Mu) =$ \{\rm idt., c-ref.\}, 
and $\Omega:  \mathrm{Aff}(\Mu) \to \mathrm{Out}(\Mu)$ is an isomorphism. 
\end{lemma}
\begin{proof} $\mathrm{Sym}(\Mu)$ = \{\rm idt., c-ref.\}, 
since a symmetry of $E^2/\Mu$ fixes the cone point and the corner point. 
We have that $\mathrm{Sym}(\Mu) = \mathrm{Aff}(\Mu)$ and $\Omega:  \mathrm{Aff}(\Mu) \to \mathrm{Out}(\Mu)$ is an isomorphism as in Lemma \ref{L1}. 
\end{proof}

\begin{table}  
\begin{tabular}{llllll}
no. & fibers & grp. & quotients &  group action & classifying pair \\
\hline 
100 & $(4\!\ast\! 2, \mathrm{O})$   & $C_1$ & $(4{\ast}2, \mathrm{O})$ & (idt., idt.) & \{idt., idt.\} \\
108 & $(4\!\ast\! 2, \mathrm{O})$   & $C_2$ & $(\ast 442, \mathrm{O})$ & (c-ref., 2-rot.) & \{c-ref., c-ref.\} \\
125 & $(4\!\ast\! 2, \mathrm{O})$  & $C_2$ & $(\ast 442, \mathrm{I})$ & (c-ref., ref.) & \{c-ref., c-ref.\}  \\
127 & $(4\!\ast\! 2, \mathrm{I})$    & $C_1$ & $(4{\ast}2, \mathrm{I})$ & (idt., idt.) & \{idt., idt.\} \\
140 & $(4\!\ast\! 2, \mathrm{I})$    & $D_1$ & $(\ast 442, \mathrm{I})$ & (c-ref., ref.) & \{idt., c-ref.\} 
\end{tabular}

\medskip
\caption{The classification of the co-Seifert fibrations of 3-space groups 
whose generic fiber is of type $4\hbox{$\ast$}2$ with IT number 12}\label{T6}
\end{table}

The derivation of Table \ref{T6} for fiber of type $4\hbox{$\ast$}2$ is the same as the derivation of Table \ref{T3}  
for fiber of type $3\!\ast\! 3$. 

\section{Generic Fiber $\ast 442$ ($45^\circ - 45^\circ$ right triangle) with IT number 11} 

The 2-space group $\Mu$ with IT number 11 is $\ast 442$ in Conway's notation or $p4m$ in IT notation. 
See space group 3/2/1/1 in Table 1A of \cite{B-Z} for the standard affine representation of $\Mu$. 
The flat orbifold $E^2/\Mu$ is a $45^\circ - 45^\circ$ right triangle.  
Let t-ref.\ denote the {\it triangle reflection} of $E^2/\Mu$. 

\begin{lemma}\label{L7} 
If $\Mu$ is the 2-space group $\ast 442$, 
then $\mathrm{Sym}(\Mu) = \mathrm{Aff}(\Mu) =$ \{\rm idt., t-ref.\}, 
and $\Omega:  \mathrm{Aff}(\Mu) \to \mathrm{Out}(\Mu)$ is an isomorphism.  
\end{lemma}
\begin{proof} $\mathrm{Sym}(\Mu)$ = \{\rm idt., t-ref.\}, 
since a symmetry of $E^2/\Mu$ fixes the $90^\circ$ corner point 
and permute the other two corner points. 
We have that $\mathrm{Sym}(\Mu) = \mathrm{Aff}(\Mu)$ and $\Omega:  \mathrm{Aff}(\Mu) \to \mathrm{Out}(\Mu)$ is an isomorphism as in Lemma \ref{L1}. 
\end{proof}

\begin{table}  
\begin{tabular}{rlllll}
no. & fibers & grp. & quotients &  group action & classifying pair \\
\hline 
  99 & $(\ast 442, \mathrm{O})$ & $C_1$ & $(\ast 442, \mathrm{O})$ & (idt., idt.) &  \{idt., idt.\} \\
107 & $(\ast 442, \mathrm{O})$ & $C_2$ & $(\ast 442, \mathrm{O})$ & (t-ref., 2-rot.) & \{t-ref., t-ref.\} \\
123 & $(\ast 442, \mathrm{I})$   & $C_1$ & $(\ast 442, \mathrm{I})$ & (idt., idt.) &  \{idt., idt.\} \\
129 & $(\ast 442, \mathrm{O})$ & $C_2$ & $(\ast 442, \mathrm{I})$ & (t-ref., ref.) & \{t-ref., t-ref.\} \\
139 & $(\ast 442, \mathrm{I})$   & $D_1$ & $(\ast 442, \mathrm{I})$ & (t-ref., ref.) & \{idt., t-ref.\}
\end{tabular}

\medskip
\caption{The classification of the co-Seifert fibrations of 3-space groups 
whose generic fiber is of type $\ast 442$ with IT number 11}\label{T7}
\end{table}

We represent $\mathrm{Out}(\Mu)$ by $\mathrm{Sym}(\Mu)$ via the isomorphism 
$\Omega:  \mathrm{Sym}(\Mu) \to \mathrm{Out}(\Mu)$. 
The set $\mathrm{Isom}(C_\infty, \Mu)$ consists of two elements corresponding 
to the pairs of inverse elements \{idt., idt.\} and \{t-ref., t-ref.\} of $\mathrm{Sym}(\Mu)$ 
by Theorem 7 of \cite{R-T-Co}. 
The corresponding affine equivalence classes of co-Seifert fibrations correspond 
to the first two rows of Table \ref{T7}. 

The set $\mathrm{Isom}(D_\infty, \Mu)$ consists of three elements corresponding 
to the pairs of elements \{idt., idt.\}, \{t-ref., t-ref.\}, \{idt., t-ref.\} of order 1 or 2 of $\mathrm{Sym}(\Mu)$ 
by Theorem 8 of \cite{R-T-Co}. 
The corresponding affine equivalence class of co-Seifert fibrations corresponds  
to the last three rows of Table \ref{T7}. 

\section{Generic Fiber $442$ (turnover) with IT number 10} 

The 2-space group $\Mu$ with IT number 10 is $442$ in Conway's notation or $p4$ in IT notation. 
See space group 3/1/1/1 in Table 1A of \cite{B-Z} for the standard affine representation of $\Mu$. 
The  flat orbifold $E^2/\Mu$ is a turnover with 3 cone points 
obtained by gluing together two congruent $45^\circ-45^\circ$ right triangles along their boundaries. 
The $442$ turnover is orientable. 
The symmetry group of this orbifold is a dihedral group of order 4 
consisting of the identity symmetry, the {\it halfturn}  2-rot., 
the {\it central reflection} c-ref.\,\,between the two triangles, and 
the {\it triangle reflection} t-ref. 

\begin{lemma}\label{L8} 
If $\Mu$ is the 2-space group $442$, 
then  $\mathrm{Sym}(\Mu)$ is the dihedral group \{\rm idt., 2-rot, c-ref., t-ref.\}, 
and $\mathrm{Sym}(\Mu) = \mathrm{Aff}(\Mu)$, 
and $\Omega:  \mathrm{Aff}(\Mu) \to \mathrm{Out}(\Mu)$ is an isomorphism.  
\end{lemma}
\begin{proof} $\mathrm{Sym}(\Mu)$ = \{\rm idt., 2-rot, c-ref., t-ref.\}, 
since a symmetry of $E^2/\Mu$ fixes the $180^\circ$ cone point and permutes  
the other two cone points. 
We have that $\mathrm{Sym}(\Mu) = \mathrm{Aff}(\Mu)$ and $\Omega:  \mathrm{Aff}(\Mu) \to \mathrm{Out}(\Mu)$ is an isomorphism as in Lemma \ref{L1}. 
\end{proof}

\begin{table} 
\begin{tabular}{rlllll}
no. & fibers & grp. & quotients &  structure group action & classifying pair \\
\hline 
75 & $(442, \mathrm{O})$ & $C_1$ & $(442, \mathrm{O})$ & (idt., idt.) & \{idt., idt.\} \\
79 & $(442, \mathrm{O})$ & $C_2$ & $(442, \mathrm{O})$ & (2-rot., 2-rot.) & \{2-rot., 2-rot.\} \\
83 & $(442, \mathrm{I})$  & $C_1$ & $(442, \mathrm{I})$ & (idt., idt.) & \{idt., idt.\}  \\
85 & $(442, \mathrm{O})$& $C_2$ & $(442, \mathrm{I})$ & (2-rot., ref.) & \{2-rot., 2-rot.\} \\
87 & $(442, \mathrm{I})$  & $D_1$ & $(442, \mathrm{I})$ & (2-rot., ref.) &  \{idt., 2-rot.\} \\
89 & $(442, \mathrm{O})$ & $C_2$ & $(\ast 442, \mathrm{I})$ & (c-ref., ref.) & \{c-ref., c-ref.\} \\
90 & $(442, \mathrm{O})$ & $C_2$ & $(4{\ast}2, \mathrm{I})$ & (t-ref., ref.) & \{t-ref., t-ref.\} \\
97 & $(442, \mathrm{O})$ & $D_2$ & $(\ast 442, \mathrm{I})$ & (c-ref., ref.), (2-rot., 2-rot.) & \{c-ref., t-ref.\} \\
103 & $(442, \mathrm{O})$& $C_2$ & $(\ast 442, \mathrm{O})$ & (c-ref., 2-rot.) & \{c-ref., c-ref.\} \\
104 & $(442, \mathrm{O})$& $C_2$ & $(4{\ast}2, \mathrm{O})$ & (t-ref., 2-rot.) & \{t-ref., t-ref.\} \\
124 & $(442, \mathrm{I})$  & $D_1$ & $(\ast 442, \mathrm{I})$ & (c-ref., ref.) &  \{idt., c-ref.\} \\
126 & $(442, \mathrm{O})$ & $D_2$ & $(\ast 442, \mathrm{I})$ & (c-ref., ref.), (t-ref., 2-rot.) & \{c-ref., 2-rot.\} \\
128 & $(442, \mathrm{I})$   & $D_1$ & $(4{\ast}2, \mathrm{I})$ & (t-ref., ref.) & \{idt., t-ref.\} \\
130 & $(442, \mathrm{O})$ & $D_2$ & $(\ast 442, \mathrm{I})$ & (t-ref., ref.), (c-ref., 2-rot.) & \{t-ref., 2-rot.\}
\end{tabular}

\medskip
\caption{The classification of the co-Seifert fibrations of 3-space groups 
whose generic fiber is of type $442$ with IT number 10}\label{T8}
\end{table}

We represent $\mathrm{Out}(\Mu)$ by $\mathrm{Sym}(\Mu)$ via the isomorphism 
$\Omega:  \mathrm{Sym}(\Mu) \to \mathrm{Out}(\Mu)$. 
The set $\mathrm{Isom}(C_\infty, \Mu)$ consists of four elements corresponding 
to the pairs of inverse elements \{idt., idt.\}, \{2-rot., 2-rot.\}, \{c-ref., c-ref.\}, \{t-ref., t-ref.\} of $\mathrm{Sym}(\Mu)$ 
by Theorem 7 of \cite{R-T-Co}. 

The set $\mathrm{Isom}(D_\infty, \Mu)$ consists of ten elements corresponding 
to the pairs of elements \{idt., idt.\}, \{2-rot., 2-rot.\}, \{idt., 2-rot.\}, \{c-ref., c-ref.\}, \{t-ref., t-ref.\}, \{c-ref., t-ref.\}, 
\{idt., c-ref.\}, \{c-ref., 2-rot.\}, \{idt., t-ref.\}, \{t-ref., 2-rot.\} of order 1 or 2 of $\mathrm{Sym}(\Mu)$ 
by Theorem 8 of \cite{R-T-Co}. 

\section{Generic Fiber $2{\ast}22$ (bonnet) with IT number 9} 

The 2-space group $\Mu$ with IT number 9 is $2{\ast}22$ in Conway's notation or $cmm$ in IT notation. 
See space group 2/2/2/1 in Table 1A of \cite{B-Z} for the standard affine representation of $\Mu$. 
The flat orbifold $E^2/\Mu$ is a {\it bonnet}. 
The most symmetric bonnet is the {\it square bonnet} 
obtained by gluing together two congruent squares along the union of two adjacent sides. 
This orbifold has one $180^\circ$ cone point and two $90^\circ$ corner points. 
The symmetry group of this orbifold is a dihedral group of order 4 
consisting of the identity symmetry, the {\it central reflection} c-ref.\ between the two squares, 
the {\it diagonal reflection} d-ref., and the {\it halfturn} 2-rot.
Both c-ref.\ and  2-rot.\ transpose the two corner points of the square bonnet, 
whereas  d-ref.\ fixes each of the corner points of the square bonnet. 

\begin{lemma}\label{L9} 
If $\Mu$ is the 2-space group $2{\ast}22$ and $E^2/\Mu$ is a square bonnet, 
then  $\mathrm{Sym}(\Mu)$ is the dihedral group \{\rm idt., 2-rot, c-ref., d-ref.\}, 
and $\mathrm{Sym}(\Mu) = \mathrm{Aff}(\Mu)$, 
and $\Omega:  \mathrm{Aff}(\Mu) \to \mathrm{Out}(\Mu)$ is an isomorphism. 
\end{lemma}
\begin{proof} $\mathrm{Sym}(\Mu)$ = \{\rm idt., 2-rot, c-ref., d-ref.\}, 
since a symmetry of $E^2/\Mu$ fixes the cone point 
and permutes the corner points. 
We have that $\mathrm{Sym}(\Mu) = \mathrm{Aff}(\Mu)$ and $\Omega:  \mathrm{Aff}(\Mu) \to \mathrm{Out}(\Mu)$ is an isomorphism as in Lemma \ref{L1}. 
\end{proof}

\begin{table} 
\begin{tabular}{rlllll}
no. & fibers & grp. & quotients &  structure group action & classifying pair \\
\hline 
35 & $(2{\ast}22, \mathrm{O})$ & $C_1$ & $(2{\ast}22, \mathrm{O})$ & (idt., idt.) & \{idt., idt.\} \\
42 & $(2{\ast}22, \mathrm{O})$ & $C_2$ & $(\ast 2222, \mathrm{O})$ & (c-ref., 2-rot.) & \{c-ref., c-ref.\} \\
65 & $(2{\ast}22, \mathrm{I})$  & $C_1$ & $(2{\ast}22, \mathrm{I})$ & (idt., idt.) & \{idt., idt.\}  \\
67 & $(2{\ast}22, \mathrm{O})$ & $C_2$ & $(\ast 2222, \mathrm{I})$ & (c-ref., ref.) & \{c-ref., c-ref.\} \\
69 & $(2{\ast}22, \mathrm{I})$   & $D_1$ & $(\ast 2222, \mathrm{I})$ &  (c-ref., ref.) & \{idt., c-ref.\} \\
101 & $(2{\ast}22, \mathrm{O})$ & $C_2$ & $(\ast 442, \mathrm{O})$ & (d-ref., 2-rot.) & \{d-ref., d-ref.\} \\
102 & $(2{\ast}22, \mathrm{O})$ & $C_2$ & $(4{\ast}2, \mathrm{O})$ & (2-rot., 2-rot.) & \{2-rot., 2-rot.\} \\
111 & $(2{\ast}22, \mathrm{O})$ & $C_2$ & $(\ast 442, \mathrm{I})$ & (d-ref., ref.) & \{d-ref., d-ref.\} \\
113 & $(2{\ast}22, \mathrm{O})$ & $C_2$ & $(4{\ast}2, \mathrm{I})$ & (2-rot., ref.) & \{2-rot., 2-rot.\} \\
121 & $(2{\ast}22, \mathrm{O})$ & $D_2$ & $(\ast 442, \mathrm{I})$ & (d-ref., ref.), \{c-ref., 2-rot.\} & \{d-ref., 2-rot.\} \\
132 & $(2{\ast}22, \mathrm{I})$  & $D_1$ & $(\ast 442, \mathrm{I})$ & (d-ref., ref.) &  \{idt., d-ref.\} \\
134 & $(2{\ast}22, \mathrm{O})$ & $D_2$ & $(\ast 442, \mathrm{I})$ & (c-ref., ref.), (2-rot., 2-rot.) & \{c-ref., d-ref.\} \\
136 & $(2{\ast}22, \mathrm{I})$  & $D_1$ & $(4{\ast}2, \mathrm{I})$ & (2-rot., ref.) & \{idt., 2-rot.\} \\
138 & $(2{\ast}22, \mathrm{O})$ & $D_2$ & $(\ast 442, \mathrm{I})$ & (c-ref., ref.), (d-ref., 2-rot.) &  \{c-ref., 2-rot.\}
\end{tabular}

\medskip
\caption{The classification of the co-Seifert fibrations of 3-space groups 
whose generic fiber is of type $2{\ast}22$ with IT number 9}\label{T9}
\end{table}

The derivation of Table \ref{T9} for fiber of type $2{\ast}22$ is similar to the derivation of Table \ref{T8} for fiber of type $442$. 

\section{Generic Fiber $22\times$ (projective pillow) with IT number 8} 

The 2-space group $\Mu$ with IT number 8 is $22\times$ in Conway's notation or $pgg$ in IT notation. 
See space group 2/2/1/3 in Table 1A of \cite{B-Z} for the standard affine representation of $\Mu$. 
The flat orbifold $E^2/\Mu$ is a {\it projective pillow}. 
The most symmetric projective pillow is the {\it square projective pillow} obtained by gluing the opposite sides 
of a square $\Box$ by glide reflections with axes the lines joining the midpoints of opposite sides of $\Box$. 
This orbifold has two $180^\circ$ cone points represented by diagonally opposite vertices of $\Box$, 
and a {\it center point} represented by the center of $\Box$. 
The symmetry group of this orbifold is a dihedral group of order 8, 
represented by the symmetry group of $\Box$, 
 consisting of the identity symmetry, two {\it midline reflections}, m-ref. and m-ref.$'$, 
two {\it diagonal reflections}, d-ref.\ and d-ref.$'$, a {\it halfturn} 2-rot., 
and two order 4 rotations, 4-rot.\ and 4-rot.$^{-1}$, with 4-rot.\ = (d-ref.)(m-ref.).  
There are three conjugacy classes of order 2 symmetries, the classes represented by m-ref., 2-rot., and d-ref. 
There is one conjugacy class of order 4 symmetries. 
There are two conjugacy classes of dihedral subgroups of order 4, 
the group generated by the midline reflections, 
and the group generated by the diagonal reflections. 

\begin{lemma}\label{L10} 
If $\Mu$ is the 2-space group $22\times$ and $E^2/\Mu$ is a square projective pillow, 
then $\mathrm{Sym}(\Mu)$ is the dihedral group $\langle$\rm m-ref., d-ref.$\rangle$ of order 8, 
and $\mathrm{Sym}(\Mu) = \mathrm{Aff}(\Mu)$, 
and $\Omega:  \mathrm{Aff}(\Mu) \to \mathrm{Out}(\Mu)$ is an isomorphism.  
\end{lemma}
\begin{proof} $\mathrm{Sym}(\Mu)$ = $\langle$\rm m-ref., d-ref.$\rangle$, 
since a symmetry of $E^2/\Mu$ fixes the center point and permutes the cone points. 
We have that $\mathrm{Sym}(\Mu) = \mathrm{Aff}(\Mu)$ and $\Omega:  \mathrm{Aff}(\Mu) \to \mathrm{Out}(\Mu)$ is an isomorphism as in Lemma \ref{L1}. 
\end{proof}

\medskip
\begin{table}  
\begin{tabular}{rlllll}
no. & fibers & grp. & quotients &  structure group action & classifying pair \\
\hline 
  32 & $(22\times, \mathrm{O})$ & $C_1$ & $(22\times, \mathrm{O})$ & (idt., idt.) & \{idt., idt.\} \\
  41 & $(22\times, \mathrm{O})$ & $C_2$ & $(22\ast, \mathrm{O})$ & (m-ref., 2-rot.) &  \{m-ref., m-ref.\} \\
  45 & $(22\times, \mathrm{O})$& $C_2$ & $(2{\ast}22, \mathrm{O})$ & (2-rot., 2-rot.) & \{2-rot., 2-rot.\} \\
  50 & $(22\times, \mathrm{O})$ & $C_2$ & $(2{\ast}22, \mathrm{I})$ & (2-rot., ref.) & \{2-rot., 2-rot.\} \\
  54 & $(22\times, \mathrm{O})$ & $C_2$ & $(22\ast, \mathrm{I})$ & (m-ref., ref.)  &  \{m-ref., m-ref.\} \\
  55 & $(22\times, \mathrm{I})$  & $C_1$ & $(22\times, \mathrm{I})$ & (idt., idt.) & \{idt., idt.\}  \\
  64 & $(22\times, \mathrm{I})$  & $D_1$ & $(22\ast, \mathrm{I})$ & (m-ref., ref.)& \{idt., m-ref.\} \\
  68 & $(22\times, \mathrm{O})$ & $D_2$ & $(\ast 2222, \mathrm{I})$ & (m-ref., ref.), (m-ref.$'$, 2-rot.) & \{m-ref., 2-rot.\} \\
  72 & $(22\times, \mathrm{I})$   & $D_1$ & $(2{\ast}22, \mathrm{I})$ & (2-rot., ref.) &  \{idt., 2-rot.\} \\
  73 & $(22\times, \mathrm{O})$ & $D_2$ & $(\ast 2222, \mathrm{I})$ & (m-ref., ref.), (2-rot., 2-rot.) & \{m-ref., m-ref.$'$\} \\
106 & $(22\times, \mathrm{O})$  & $C_2$ & $(4{\ast}2, \mathrm{O})$ & (d-ref., 2-rot.) & \{d-ref., d-ref.\} \\
110 & $(22\times, \mathrm{O})$  & $C_4$ & $(4{\ast}2, \mathrm{O})$ & (4-rot., 4-rot.) & \{4-rot., 4-rot.$^{-1}$\} \\
117 & $(22\times, \mathrm{O})$  & $C_2$ & $(4{\ast}2, \mathrm{I})$ & (d-ref., ref.) & \{d-ref., d-ref.\} \\
120 & $(22\times, \mathrm{O})$ & $D_2$ & $(\ast 442, \mathrm{I})$ & (d-ref., ref.), (2-rot., 2-rot.) & \{d-ref., d-ref.$'$\} \\
133 & $(22\times, \mathrm{O})$ & $D_2$ & $(\ast 442, \mathrm{I})$ & (d-ref., ref.), (d-ref.$'$, 2-rot.) &  \{d-ref., 2-rot.\} \\
135 & $(22\times, \mathrm{I})$   & $D_1$ & $(4{\ast}2, \mathrm{I})$ & (d-ref., ref.) & \{idt., d-ref.\} \\
142 & $(22\times, \mathrm{O})$ & $D_4$ & $(\ast 442, \mathrm{I})$ & (d-ref., ref.), (4-rot., 4-rot.) & \{d-ref., m-ref.\}
\end{tabular}

\medskip
\caption{The classification of the co-Seifert fibrations of 3-space groups 
whose generic fiber is of type $22\times$ with IT number 8}\label{T10}
\end{table}

We represent $\mathrm{Out}(\Mu)$ by $\mathrm{Sym}(\Mu)$ via the isomorphism 
$\Omega:  \mathrm{Sym}(\Mu) \to \mathrm{Out}(\Mu)$. 
The set $\mathrm{Isom}(C_\infty, \Mu)$ consists of five elements corresponding 
to the pairs of inverse elements \{idt., idt.\}, \{m-ref., m-ref.\}, \{2-rot., 2-rot.\}, \{d-ref., d-ref.\},  \{4-rot., 4-rot.$^{-1}$\} 
of $\mathrm{Sym}(\Mu)$ by Theorem 7 of \cite{R-T-Co}. 

The set $\mathrm{Isom}(D_\infty, \Mu)$ consists of twelve elements corresponding 
to the pairs of elements \{2-rot., 2-rot.\},  \{m-ref., m-ref.\}, \{idt., idt.\}, \{idt., m-ref.\},  \{m-ref., 2-rot.\}, \{idt., 2-rot.\}, 
\{m-ref., m-ref.$'$\}, \{d-ref., d-ref.\}, \{d-ref., d-ref.$'$\},  \{d-ref., 2-rot.\}, \{idt., d-ref.\}, \{d-ref., m-ref.\}  
of order 1 or 2 of $\mathrm{Sym}(\Mu)$ by Theorem 8 of \cite{R-T-Co}. 

\section{Generic Fiber $22\ast$ (pillowcase) with IT number 7} 

The 2-space group $\Mu$ with IT number 7 is $22\ast$ in Conway's notation or $pmg$ in IT notation. 
See space group 2/2/1/2 in Table 1A of \cite{B-Z} for the standard affine representation of $\Mu$. 
The flat orbifold $E^2/\Mu$ is a {\it pillowcase}  
obtained by gluing together two congruent rectangles along the union of three of their sides. 
This orbifold has two $180^\circ$ cone points. 
The symmetry group of this orbifold is a dihedral group of order 4 
consisting of the identity symmetry, the {\it central reflection} c-ref.\ between the two rectangles, 
the {\it halfturn} 2-rot.,  and the {\it midline reflection} m-ref.

\begin{lemma}\label{L11} 
If $\Mu$ is the 2-space group $22\ast$, 
then $\mathrm{Sym}(\Mu)$ is the dihedral group \{\rm idt., c-ref., m-ref., 2-rot.\}, 
and $\mathrm{Sym}(\Mu) = \mathrm{Aff}(\Mu)$,  
and $\Omega:  \mathrm{Aff}(\Mu) \to \mathrm{Out}(\Mu)$ is an isomorphism.  
\end{lemma}
\begin{proof} $\mathrm{Sym}(\Mu)$ = \{\rm idt., c-ref., m-ref., 2-rot\}, 
since a symmetry of $E^2/\Mu$ permutes the cone points. 
We have that $\mathrm{Sym}(\Mu) = \mathrm{Aff}(\Mu)$ and $\Omega:  \mathrm{Aff}(\Mu) \to \mathrm{Out}(\Mu)$ is an isomorphism as in Lemma \ref{L1}. 
\end{proof}

\begin{table} 
\begin{tabular}{llllll}
no. & fibers & grp. & quotients &  structure group action & classifying pair \\
\hline 
28 & $(22\ast, \mathrm{O})$ & $C_1$ & $(22\ast, \mathrm{O})$ & (idt., idt.) & \{idt., idt.\} \\
39 & $(22\ast, \mathrm{O})$ & $C_2$ & $(\ast 2222, \mathrm{O})$ & (c-ref., 2-rot.) & \{c-ref., c-ref.\} \\
40 & $(22\ast, \mathrm{O})$ & $C_2$ & $(22\ast, \mathrm{O})$ & (2-rot., 2-rot.) & \{2-rot., 2-rot.\} \\
46 & $(22\ast, \mathrm{O})$ & $C_2$ & $(2{\ast}22, \mathrm{O})$ & (m-ref., 2-rot.) & \{m-ref., m-ref.\} \\
49 & $(22\ast, \mathrm{O})$ & $C_2$ & $(\ast 2222, \mathrm{I})$ & (c-ref., ref.) & \{c-ref., c-ref.\} \\
51 & $(22\ast, \mathrm{I})$   & $C_1$ & $(22\ast, \mathrm{I})$ & (idt., idt.) & \{idt., idt.\}  \\ 
53 & $(22\ast, \mathrm{O})$ & $C_2$ & $(2{\ast}22, \mathrm{I})$ & (m-ref., ref.) & \{m-ref., m-ref.\} \\
57 & $(22\ast, \mathrm{O})$ & $C_2$ & $(22\ast, \mathrm{I})$ & (2-rot., ref.) & \{2-rot., 2-rot.\}  \\
63 & $(22\ast, \mathrm{I})$   & $D_1$ & $(22\ast, \mathrm{I})$ & (2-rot., ref.) & \{idt., 2-rot.\}\\
64 & $(22\ast, \mathrm{O})$ & $D_2$ & $(\ast 2222, \mathrm{I})$ & (m-ref., ref.), (c-ref., 2-rot.) & \{m-ref., 2-rot.\}\\
66 & $(22\ast, \mathrm{O})$ & $D_2$ & $(\ast 2222, \mathrm{I})$ & (c-ref., ref.), (2-rot., 2-rot.) & \{c-ref., m-ref.\} \\
67 & $(22\ast, \mathrm{I})$   & $D_1$ & $(\ast 2222, \mathrm{I})$ & (c-ref., ref.) & \{idt., c-ref.\} \\
72 & $(22\ast, \mathrm{O})$ & $D_2$ & $(\ast 2222, \mathrm{I})$ & (c-ref., ref.), (m-ref., 2-rot.) &  \{c-ref., 2-rot.\} \\ 
74 & $(22\ast, \mathrm{I})$   & $D_1$ & $(2{\ast}22, \mathrm{I})$ & (m-ref., ref.) & \{idt., m-ref.\} 
\end{tabular}

\medskip
\caption{The classification of the co-Seifert fibrations of 3-space groups 
whose generic fiber is of type $22\ast$ with IT number 7}\label{T11}
\end{table}

The derivation of Table \ref{T11} for fiber of type $22\ast$ is similar to the derivation of Table \ref{T8} for fiber of type $442$. 

\section{Generic Fiber $\ast 2222$ (rectangle) with IT number 6} 

The 2-space group $\Mu$ with IT number 6 is $\ast 2222$ in Conway's notation or $pmm$ in IT notation. 
See space group 2/2/1/1 in Table 1A of \cite{B-Z} for the standard affine representation of $\Mu$. 
The flat orbifold $E^2/\Mu$ is a rectangle. 
A rectangle has four $90^\circ$ corner points. 
The  most symmetric rectangle is a square $\Box$. 
The symmetry group of $\Box$ is a dihedral group of order 8 
consisting of the identity symmetry, two {\it midline reflections}, m-ref.\ and m-ref.$'$, 
two {\it diagonal reflections},  d-ref.\ and d-ref.$'$, a {\it halfturn} 2-rot., 
and two order 4 {\it rotations}, 4-rot.\,\,and 4-rot.$^{-1}$, with 4-rot.\ = (d-ref.)(m-ref.). 
There are three conjugacy classes of order 2 symmetries, the classes represented by m-ref., 2-rot., and d-ref. 
There is one conjugacy class of order 4 symmetries. 
There are two conjugacy classes of dihedral subgroups of order 4, 
the group generated by the midline reflections, 
and the group generated by the diagonal reflections. 

\begin{lemma}\label{L12} 
If $\Mu$ is the 2-space group $\ast 2222$ and $E^2/\Mu$ is a square, 
then $\mathrm{Sym}(\Mu)$ is the dihedral group $\langle$\rm m-ref., d-ref.$\rangle$ of order 8, 
and $\mathrm{Sym}(\Mu) = \mathrm{Aff}(\Mu)$, 
and $\Omega:  \mathrm{Aff}(\Mu) \to \mathrm{Out}(\Mu)$ is an isomorphism.  
\end{lemma}
\begin{proof} $\mathrm{Sym}(\Mu)$ = $\langle$\rm m-ref., d-ref.$\rangle$, 
since a symmetry of $E^2/\Mu$ permutes the corner points. 
We have that $\mathrm{Sym}(\Mu) = \mathrm{Aff}(\Mu)$ and $\Omega:  \mathrm{Aff}(\Mu) \to \mathrm{Out}(\Mu)$ is an isomorphism as in Lemma \ref{L1}. 
\end{proof}

\begin{table}  
\begin{tabular}{rlllll}
no. & fibers & grp. & quotients &  structure group action & classifying pair \\
\hline 
  25 & $(\ast 2222, \mathrm{O})$ & $C_1$ & $(\ast 2222, \mathrm{O})$ & (idt., idt.) & \{idt., idt.\} \\
  38 & $(\ast 2222, \mathrm{O})$ & $C_2$ & $(\ast 2222, \mathrm{O})$ & (m-ref., 2-rot.) & \{m-ref., m-ref.\} \\
  44 & $(\ast 2222, \mathrm{O})$ & $C_2$ & $(2{\ast}22, \mathrm{O})$ & (2-rot., 2-rot.) & \{2-rot., 2-rot.\} \\
  47 & $(\ast 2222, \mathrm{I})$   & $C_1$ & $(\ast 2222, \mathrm{I})$ & (idt., idt.) & \{idt., idt.\} \\
  51 & $(\ast 2222, \mathrm{O})$ & $C_2$ & $(\ast 2222, \mathrm{I})$ & (m-ref., ref.) & \{m-ref., m-ref.\} \\
  59 & $(\ast 2222, \mathrm{O})$ & $C_2$ & $(2{\ast}22, \mathrm{I})$ & (2-rot., ref.) & \{2-rot., 2-rot.\}  \\
  63 & $(\ast 2222, \mathrm{O})$ & $D_2$ & $(\ast 2222, \mathrm{I})$ & (m-ref., ref.),  (m-ref.$'$, 2-rot.) &  \{m-ref., 2-rot.\} \\
  65 & $(\ast 2222, \mathrm{I})$   & $D_1$ & $(\ast 2222, \mathrm{I})$ &  (m-ref., ref.) & \{idt., m-ref.\} \\
  71 & $(\ast 2222, \mathrm{I})$   & $D_1$ & $(2{\ast}22, \mathrm{I})$ & (2-rot., ref.) &  \{idt., 2-rot.\} \\
  74 & $(\ast 2222, \mathrm{O})$ & $D_2$ & $(\ast 2222, \mathrm{I})$ & (m-ref., ref.), (2-rot., 2-rot.) & \{m-ref., m-ref.$'$\} \\
105 & $(\ast 2222, \mathrm{O})$ & $C_2$ & $(\ast 442, \mathrm{O})$ & (d-ref., 2-rot.) & \{d-ref., d-ref.\} \\
109 & $(\ast 2222, \mathrm{O})$ & $C_4$ & $(4{\ast}2, \mathrm{O})$ & (4-rot., 4-rot.) & \{4-rot., 4-rot.$^{-1}$\} \\
115 & $(\ast 2222, \mathrm{O})$ & $C_2$ & $(\ast 442, \mathrm{I})$ & (d-ref., ref.) & \{d-ref., d-ref.\} \\
119 & $(\ast 2222, \mathrm{O})$ & $D_2$ & $(\ast 442, \mathrm{I})$ & (d-ref., ref.), (2-rot., 2-rot.) & \{d-ref., d-ref.$'$\} \\
131 & $(\ast 2222, \mathrm{I})$   & $D_1$ & $(\ast 442, \mathrm{I})$ & (d-ref., ref.) & \{idt., d-ref.\} \\
137 & $(\ast 2222, \mathrm{O})$ & $D_2$ & $(\ast 442, \mathrm{I})$ & (d-ref., ref.), (d-ref.$'$, 2-rot.) & \{d-ref., 2-rot.\} \\
141 & $(\ast 2222, \mathrm{O})$ & $D_4$ & $(\ast 442, \mathrm{I})$ & (d-ref., ref.), (4-rot., 4-rot.) &  \{d-ref., m-ref.\}
\end{tabular}

\medskip
\caption{The classification of the co-Seifert fibrations of 3-space groups 
whose generic fiber is of type $\ast 2222$ with IT number 6}\label{T12}
\end{table}

The derivation of Table \ref{T12} for fiber of type $\ast 2222$ is similar to the derivation of Table \ref{T10} for fiber of type $22\times$. 

\section{Generic Fiber $\ast\times$ (M\"obius band) with IT number 5} 

Let $\Mu = \langle e_1+I, e_2 + I, A \rangle$ where $e_1, e_2$ are the standard basis vectors of $E^2$ 
and $A$ transposes $e_1$ and $e_2$. 
Then $\Mu$ is a 2-space group with IT number 5.  
The Conway notation for $\Mu$ is $\ast\times$ and the IT notation is $cm$. 
The flat orbifold $E^2/\Mu$ is a M\"obius band. 
We have that $Z(\Mu) = \langle e_1+e_2 +I\rangle$. 

\begin{lemma}\label{L13}  
If $\Mu = \langle e_1+I, e_2 + I, A \rangle$, with $A$ transposing $e_1$ and $e_2$, and $N_A(\Mu)$ 
is the normalizer of $\Mu$ in $\mathrm{Aff}(E^2)$, then 
$$N_A(\Mu) = \{b+B: b_1-b_2 \in \integers \ \, \hbox{and}\ \, B \in \langle -I, A\rangle\}.$$
\end{lemma}
\begin{proof}
Observe that $b+B \in N_A(\Mu)$ if and only if $B\in N_A(\langle e_1+I, e_2+I\rangle)$ 
and $(b+B)A(b+B)^{-1}\in \Mu$.  
Now $B \in N_A(\langle e_1+I, e_2+I\rangle)$ if and only if $B \in \mathrm{GL}(2,\integers)$.
As $(b+B)A(b+B)^{-1} = b-BAB^{-1}b+BAB^{-1}$, we have that $(b+B)A(b+B)^{-1}\in \Mu$ 
if and only if $BAB^{-1} = A$ and $b-Ab \in \integers^2$. 
Hence $b+B \in N_A(\Mu)$ if and only if $B \in \langle -I, A\rangle$ and $b_1-b_2 \in \integers$. 
\end{proof}

A fundamental polygon for the action of $\Mu$ on $E^2$ is the $45^\circ -45^\circ$ right triangle 
with vertices $v_1=(0,0), v_2=(1,0), v_3 = (1,1)$. 
The reflection $A$ fixes the hypotenuse  $[v_1,v_3]$ pointwise.  
The glide reflection $e_1+A$ maps the side $[v_1,v_2]$ to the side $[v_2,v_3]$. 
The {\it boundary} of the M\"obius band $E^2/\Mu$ is represented by the hypotenuse $[v_1,v_3]$, 
and the {\it central circle} of $E^2/\Mu$ is represented by the line segment $[(1/2,0), (1,1/2)]$ 
joining the midpoints of the two short sides of the triangle.  

Restricting a symmetry of the M\"obius band $E^2/\Mu$ to a symmetry of its boundary circle 
gives an isomorphism from the symmetry group of the M\"obius band to the symmetry group of its boundary circle. 
There are two conjugacy classes of elements of order 2, 
the class of the reflection c-ref.\  = $(e_1/2 +e_2/2+I)_\star$ in the central circle of the M\"obius band, 
and the class of a halfturn 2-rot.\ = $(-I)_\star$ about the  point $\Mu(1/2,0)$ on the central circle. 
Note that c-ref.\ restricts to a halfturn of the boundary circle, and 2-rot.\ restricts to a reflection of the boundary circle. 
Define 2-rot.$'$ = $(e_1/2+e_2/2-I)_\star$. 
There is one conjugacy class of dihedral symmetry groups of order 4, represented by the group \{idt., c-ref., 2-rot., 2-rot.$'$\}.

\begin{lemma}\label{L14} 
If $\Mu= \langle e_1+I, e_2 + I, A \rangle$, with $A$ transposing $e_1$ and $e_2$, 
then the Lie group $\mathrm{Sym}(\Mu)$ is isomorphic to $\mathrm{O}(2)$, 
and $\mathrm{Sym}(\Mu) = \mathrm{Aff}(\Mu)$, 
and $\Omega:  \mathrm{Aff}(\Mu) \to \mathrm{Out}(\Mu)$ maps 
the subgroup {\rm \{idt., 2-rot.\}} isomorphically onto $\mathrm{Out}(\Mu)$.  
\end{lemma}
\begin{proof} The Lie group $\mathrm{Sym}(\Mu)$ is isomorphic to $\mathrm{O}(2)$, 
since restricting a symmetry of the M\"obius band $E^2/\Mu$ to a symmetry of its boundary circle $\mathrm{O}$ is 
an isomorphism from $\mathrm{Sym}(\Mu)$ to $\mathrm{Isom}(\mathrm{O})$ by Lemma 1 of \cite{R-T-I} and Lemma \ref{L13}. 
We have that $\mathrm{Sym}(\Mu) = \mathrm{Aff}(\Mu)$ by Lemmas 1 and 7 of \cite{R-T-I} and Lemma \ref{L13}. 
The epimorphism $\Omega:  \mathrm{Aff}(\Mu) \to \mathrm{Out}(\Mu)$ maps  
the group \{idt., 2-rot.\} isomorphically onto $\mathrm{Out}(\Mu)$ by Theorem 3 of \cite{R-T-I}, 
since 2-rot.\ acts as a reflection of $\mathrm{O}$. 
\end{proof}

\begin{table}  
\begin{tabular}{rlllll}
no. & fibers & grp. & quotients &  structure group action &  classifying pair\\
\hline 
  8 & $(\ast\times, \mathrm{O})$ & $C_1$ & $(\ast\times, \mathrm{O})$ & (idt., idt.) & \{idt., idt.\} \\
12 & $(\ast\times, \mathrm{O})$ & $C_2$ & $(2{\ast}22, \mathrm{I})$ & (2-rot., ref.) &  \{2-rot., 2-rot.$\}_\ast$ \\
36 & $(\ast\times, \mathrm{O})$ & $C_2$ & $(2{\ast}22, \mathrm{O})$ & (2-rot., 2-rot.)  & \{2-rot., 2-rot.\} \\
38 & $(\ast\times, \mathrm{I})$   & $C_1$ & $(\ast\times, \mathrm{I})$ & (idt., idt.)  & \{idt., idt.\} \\
39 & $(\ast\times, \mathrm{O})$ & $C_2$ & $(\ast\ast, \mathrm{I})$ & (c-ref., ref.) & \{c-ref., c-ref.\} \\
42 & $(\ast\times, \mathrm{I})$   & $D_1$ & $(\ast\ast, \mathrm{I})$ & (c-ref., ref.) &  \{idt., c-ref.\} \\
63 & $(\ast\times, \mathrm{I})$   & $D_1$ & $(2{\ast}22, \mathrm{I})$ & (2-rot., ref.) &  \{idt., 2-rot.\} \\
64 & $(\ast\times, \mathrm{O})$ & $D_2$ & $(\ast 2222, \mathrm{I})$ & (2-rot., ref.), (2-rot.$'$, 2-rot.) &  \{2-rot., c-ref.\}
\end{tabular}

\medskip
\caption{The classification of the co-Seifert fibrations of 3-space groups 
whose generic fiber is of type $\ast\times$ with IT number 5}\label{T13}
\end{table}

We represent $\mathrm{Out}(\Mu)$ by the subgroup  \{idt., 2-rot.\} of $\mathrm{Sym}(\Mu)$ 
via the epimorphism $\Omega:  \mathrm{Sym}(\Mu) \to \mathrm{Out}(\Mu)$.
The set $\mathrm{Isom}(C_\infty, \Mu)$ consists of two elements corresponding 
to the pairs of inverse elements \{idt., idt.\} and \{2-rot., 2-rot.\} of $\mathrm{Sym}(\Mu)$ 
by Theorem 7 of \cite{R-T-Co}. 

The set $\mathrm{Isom}(D_\infty, \Mu)$ consists of six elements corresponding 
to the pairs of elements \{2-rot., 2-rot.\}, \{idt., idt.\}, \{c-ref., c-ref.\}, \{idt., c-ref.\},  \{idt., 2-rot.\}, \{2-rot., c-ref.\} of order 1 or 2 of $\mathrm{Sym}(\Mu)$ 
by Lemma 8 of \cite{R-T-Co} and Theorems 9 and 10 of \cite{R-T-Co}. 
Notice that if 2-rot.$''$ = $(xe_1+xe_2 -I)_\star$ for some $x \in \realnos$, 
then the pairs \{2-rot., 2-rot.\} and \{2-rot., 2-rot.$''$\} determine the same element of  $\mathrm{Isom}(D_\infty, \Mu)$
by Theorem 9 of \cite{R-T-Co}.
The pairs \{2-rot., c-ref.\} and \{2-rot$'$., c-ref.\} determine the same element of  $\mathrm{Isom}(D_\infty, \Mu)$, 
since they are conjugate. 

\section{Generic Fiber $\times\times$ (Klein bottle) with IT number 4} 

Let $\Mu = \langle e_1+I, e_2+I, e_1/2+A\rangle$ where $A = \mathrm{diag}(1,-1)$. 
Then $\Mu$ is a 2-space group with IT number 4. 
The Conway notation for $\Mu$ is $\times\times$ and the IT notation is $pg$. 
The flat orbifold $E^2/\Mu$ is a Klein bottle. 
We have that $Z(\Mu) = \langle e_1+I \rangle$.

\begin{lemma}\label{L15}  
If $\Mu = \langle e_1+I, e_2 + I, e_1/2+ A \rangle$, with $A = \mathrm{diag}(1,-1)$, 
and $N_A(\Mu)$ is the normalizer of $\Mu$ in $\mathrm{Aff}(E^2)$, then 
$$N_A(\Mu) = \{b+B: 2b_2 \in \integers \ \, \hbox{and}\ \, B \in \langle -I, A\rangle\}.$$
\end{lemma}
\begin{proof}
Observe that $b+B \in N_A(\Mu)$ if and only if $B\in N_A(\langle e_1+I, e_2+I\rangle)$ 
and $(b+B)(e_1/2+A)(b+B)^{-1}\in \Mu$.  
Now $B \in N_A(\langle e_1+I, e_2+I\rangle)$ if and only if $B \in \mathrm{GL}(2,\integers)$.
As $(b+B)(e_1/2+A)(b+B)^{-1} = b+Be_1/2-BAB^{-1}b+BAB^{-1}$, we have that $(b+B)(e_1/2+A)(b+B)^{-1}\in \Mu$ 
if and only if $BAB^{-1} = A$ and $b-Ab +Be_1/2-e_1/2\in \integers^2$. 
Hence $b+B \in N_A(\Mu)$ if and only if $B \in \langle -I, A\rangle$ and $2b_2 \in \integers$. 
\end{proof}

A fundamental polygon for the action of $\Mu$ on $E^2$ is the rectangle 
with vertices $v_1 = (0,-1/2), v_2 = (1/2,-1/2), v_3 = (1/2,1/2)$ and $v_4= (0,1/2)$. 
The vertical translation $e_2+I$ maps the side $[v_1,v_2]$ to the side $[v_4,v_3]$. 
The horizontal glide reflection $e_1/2 + A$ maps the side $[v_1,v_4]$ to the side $[v_2,v_3]$. 
The Klein bottle $E^2/\Mu$ has two short horizontal geodesics represented by the line segments 
$[v_1, v_2]$ and $[(0,0),(1/2,0)]$. 
The union of these two short geodesics is invariant under any symmetry of the Klein bottle. 
Therefore the horizontal quarterline geodesic, which we call the {\it central circle}, represented by the union 
of the line segments $[(0,-1/4),(1/2,-1/4)]$ and $[(0,1/4),(1/2,1/4)]$ is invariant under any symmetry of the Klein Bottle. 
The symmetry group of the Klein bottle is the direct product of the subgroup of order 2, generated by 
the reflection in the central circle, and the subgroup consisting of the horizontal translations and the reflections 
in a vertical geodesic.   The latter subgroup restricts to the symmetry group of the central circle.  

There are five conjugacy classes of symmetries of order 2, 
the class of the reflection m-ref.\ = $A_\star = (e_1/2+I)_\star$ in the short horizontal geodesics, 
the class of the vertical halfturn 2-sym.\ = $(e_2/2+I)_\star$, 
the class of the reflection c-ref.\ = $(e_2/2 + A)_\star = (e_1/2+e_2/2+I)_\star$ in the central circle,  
the class of the reflection v-ref.\ = $(-A)_\star = (e_1/2-I)_\star$ in the vertical geodesic represented by $[v_1, v_4]$, 
and the class of the halfturn 2-rot.\ = $(e_2/2-I)_\star$ around a pair of antipodal points on the central circle. 
Define v-ref.$' = (-I)_\star$ and 2-rot.$' = (e_1/2 + e_2/2 - I)_\star$.  

There are five conjugacy classes of dihedral symmetry groups of order 4, 
the classes of the groups \{idt., m-ref., 2-sym., c-ref.\},  \{idt.,  m-ref., v-ref., v-ref.$'$\},  \{idt., m-ref., 2-rot., 2-rot.$'$\}, 
 \{idt.,  2-sym., v-ref., 2-rot.$'$\}, and  \{idt.,  c-ref. v-ref., 2-rot.\}. 

\begin{lemma}\label{L16} 
If $\Mu = \langle e_1+I, e_2 + I, e_1/2+ A \rangle$, with $A = \mathrm{diag}(1,-1)$, 
then the Lie group $\mathrm{Sym}(\Mu)$ is isomorphic to $C_2\times\mathrm{O}(2)$, 
and $\mathrm{Sym}(\Mu) = \mathrm{Aff}(\Mu)$, 
and $\Omega:  \mathrm{Aff}(\Mu) \to \mathrm{Out}(\Mu)$ maps 
the subgroup {\rm \{idt., c-ref., v-ref., 2-rot.\}}  isomorphically onto $\mathrm{Out}(\Mu)$.  
Moreover $\Omega${\rm (2-sym.)} = $\Omega${\rm (c-ref.)}. 
\end{lemma}
\begin{proof} Let $\mathrm{O}$ be the central circle of the Klein bottle, 
and let $\Lambda$ be the group of all symmetries of the Klein bottle that leave each horizontal geodesic invariant. 
Restriction induces an isomorphism from $\Lambda$ to $\mathrm{Isom}(\mathrm{O})$ by Lemma 1 of \cite{R-T-I} 
and Lemma \ref{L15}. 
We have that  $\mathrm{Sym}(\Mu) = \langle$c-ref.$\rangle \times \Lambda$, 
since every symmetry of the Klein bottle leaves $\mathrm{O}$ invariant,   
and c-ref.\ commutes with every symmetry in $\Lambda$. 
Therefore $\mathrm{Sym}(\Mu)$ is isomorphic to $C_2\times \mathrm{O}(2)$. 
We have that $\mathrm{Sym}(\Mu) = \mathrm{Aff}(\Mu)$ by Lemmas 1 and 7 of \cite{R-T-I} and Lemma \ref{L15}. 
The epimorphism $\Omega:  \mathrm{Aff}(\Mu) \to \mathrm{Out}(\Mu)$ maps 
{\rm \{idt., c-ref., v-ref., 2-rot.\}} isomorphically onto $\mathrm{Out}(\Mu)$ by Theorem 3 of \cite{R-T-I}, 
since v-ref.\ acts as a reflection of $\mathrm{O}$. 
Moreover $\Omega$(2-sym.) = $\Omega$((c-ref.)(m-ref.)) = $\Omega$(c-ref.). 
\end{proof}

\medskip

\begin{table}  
\begin{tabular}{rlllll}
no. & fibers & grp. & quotients &  structure group action & classifying pair \\
\hline 
     7& $(\times\times, \mathrm{O})$ & $C_1$ & $(\times\times, \mathrm{O})$ & (idt., idt.) &  \{idt., idt.\}  \\
    9 & $(\times\times, \mathrm{O})$ & $C_2$ & $(\times\times, \mathrm{O})$ & (2-sym., 2-rot.) &  \{2-sym., 2-sym.\} \\
  13 & $(\times\times, \mathrm{O})$& $C_2$ & $(22\ast, \mathrm{I})$ & (v-ref., ref.) &  \{v-ref., v-ref.$\}_\ast$ \\
  14 & $(\times\times, \mathrm{O})$ & $C_2$ & $(22\times, \mathrm{I})$ & (2-rot., ref.) & \{2-rot., 2-rot.$\}_\ast$ \\
  15 & $(\times\times, \mathrm{O})$ & $D_2$ & $(22\ast, \mathrm{I})$ &  (v-ref., ref.), (2-sym., 2-rot.) & \{v-ref., 2-rot.$'\}_\ast$ \\
  26 & $(\times\times, \mathrm{I})$  & $C_1$ & $(\times\times, \mathrm{I})$ & (idt., idt.) &  \{idt., idt.\}  \\
  27 & $(\times\times, \mathrm{O})$ & $C_2$ & $(\ast\ast, \mathrm{I})$ & (m-ref., ref.) &  \{m-ref., m-ref.\} \\
  29 & $(\times\times, \mathrm{O})$ & $C_2$ & $(\times\times, \mathrm{I})$ & (2-sym., ref.) &  \{2-sym., 2-sym.\} \\
  29 & $(\times\times, \mathrm{O})$ & $C_2$ & $(22\ast, \mathrm{O})$ & (v-ref., 2-rot.) &  \{v-ref., v-ref.\} \\
  30 & $(\times\times, \mathrm{O})$ & $C_2$ & $(\ast\times, \mathrm{I})$ & (c-ref., ref.) &  \{c-ref., c-ref.\} \\
  33 & $(\times\times, \mathrm{O})$  & $C_2$ & $(22\times, \mathrm{O})$ & (2-rot., 2-rot.) & \{2-rot., 2-rot.\} \\
  36 & $(\times\times, \mathrm{I})$ & $D_1$ & $(\times\times, \mathrm{I})$ & (2-sym., ref.) & \{idt., 2-sym.\} \\
  37 & $(\times\times, \mathrm{O})$   & $D_2$ & $(\ast\ast, \mathrm{I})$ & (m-ref., ref.), (2-sym., 2-rot.) & \{m-ref.,  c-ref.\} \\
  39 & $(\times\times, \mathrm{I})$ & $D_1$ & $(\ast\ast, \mathrm{I})$ & (m-ref., ref.) & \{idt., m-ref.\} \\
  41 & $(\times\times, \mathrm{O})$ & $D_2$ & $(\ast\ast, \mathrm{I})$ & (c-ref., ref.), (m-ref., 2-rot.) & \{c-ref., 2-sym.\} \\
  45 & $(\times\times, \mathrm{O})$   & $D_2$ & $(\ast\ast, \mathrm{I})$ & (m-ref., ref.), (c-ref., 2-rot.) &  \{m-ref.,  2-sym.\} \\
  46 & $(\times\times, \mathrm{I})$ & $D_1$ & $(\ast\times, \mathrm{I})$ & (c-ref., ref.) & \{idt., c-ref.\} \\
  52 & $(\times\times, \mathrm{O})$ & $D_2$ & $(2{\ast}22, \mathrm{I})$ & (c-ref., ref.), (2-rot., 2-rot.) & \{c-ref., v-ref.\} \\
  54 & $(\times\times, \mathrm{O})$  & $D_2$ & $(\ast 2222, \mathrm{I})$ & (m-ref., ref.), (v-ref.$'$, 2-rot.) & \{m-ref.,  v-ref.\} \\
  56 & $(\times\times, \mathrm{O})$ & $D_2$ & $(22\ast, \mathrm{I})$ & (m-ref., ref.),  (2-rot.$'$, 2-rot.) &  \{m-ref., 2-rot.\} \\
  57 & $(\times\times, \mathrm{I})$   & $D_1$ & $(22\ast, \mathrm{I})$ & (v-ref., ref.) & \{idt., v-ref.\} \\
  60 & $(\times\times, \mathrm{O})$ & $D_2$ & $(22\ast, \mathrm{I})$ & (2-sym., ref.), (2-rot.$'$, 2-rot.) & \{2-sym., v-ref.\} \\
  60 & $(\times\times, \mathrm{O})$ & $D_2$ & $(2{\ast}22, \mathrm{I})$ & (c-ref., ref.), (v-ref., 2-rot.) & \{c-ref., 2-rot.\} \\
  61 & $(\times\times, \mathrm{O})$   & $D_2$ & $(22\ast, \mathrm{I})$ & (2-sym., ref.), (v-ref., 2-rot.) & \{2-sym., 2-rot.$'$\} \\
  62 & $(\times\times, \mathrm{I})$ & $D_1$ & $(22\times, \mathrm{I})$ & (2-rot., ref.) & \{idt., 2-rot.\}

\end{tabular}

\medskip
\caption{The classification of the co-Seifert fibrations of 3-space groups 
whose generic fiber is of type $\times\times$ with IT number 4}\label{T14}
\end{table}

We represent $\mathrm{Out}(\Mu)$ by the subgroup  \{idt., c-ref., v-ref., 2-rot.\} of $\mathrm{Sym}(\Mu)$ 
via the epimorphism $\Omega:  \mathrm{Sym}(\Mu) \to \mathrm{Out}(\Mu)$.
The set $\mathrm{Isom}(C_\infty, \Mu)$ consists of four elements corresponding 
to the pairs of inverse elements \{idt., idt.\}, \{c-ref., c-ref.\}, \{v-ref., v-ref.\}, \{2-rot., 2-rot.\} of $\mathrm{Sym}(\Mu)$ 
by Theorem 7 of \cite{R-T-Co}. 
Notice that the pairs \{2-sym., 2-sym.\} and \{c-ref., c-ref.\} determine the same element in $\mathrm{Isom}(C_\infty,\Mu)$, 
since $\Omega$(2-sym.) = $\Omega$(c-ref.).

The set $\mathrm{Isom}(D_\infty, \Mu)$ consists of twenty one elements corresponding 
to the pairs of elements \{v-ref., v-ref.\}, \{2-rot., 2-rot.\},   \{v-ref., 2-rot.$'$\}, \{idt., idt.\}, \{m-ref., m-ref.\}, \{2-sym., 2-sym.\}, 
\{c-ref., c-ref.\},  \{idt., 2-sym.\}, \{m-ref., c-ref.\}, \{idt., m-ref.\}, \{c-ref., 2-sym.\},  \{m-ref., 2-sym.\}, \{idt., c-ref.\},  
\{c-ref., v-ref.\}, \{m-ref., v-ref.\}, \{m-ref., 2-rot.\}, \{idt., v-ref.\}, \{2-sym., v-ref.\}, \{c-ref., 2-rot.\}, \{2-sym., 2-rot.$'$\}, 
\{idt., 2-rot.\} of order 1 or 2 of $\mathrm{Sym}(\Mu)$ 
by Lemma 8 of \cite{R-T-Co} and Theorems 9 and 10 of \cite{R-T-Co}. 

Notice that the pairs \{v-ref., v-ref.\} and \{v-ref., v-ref.$'$\} determine the same element of $\mathrm{Isom}(D_\infty, \Mu)$
by Theorem 9 of \cite{R-T-Co}. The same is true for \{2-rot., 2-rot.\} and \{2-rot., 2-rot.$'$\}, 
and \{v-ref., 2-rot.\} and \{v-ref., 2-rot.$'$\}. 
The pairs \{m-ref., v-ref.\} and \{m-ref., v-ref.$'$\}, and \{m-ref., 2-rot.\} and \{m-ref., 2-rot.$'$\} 
determine the same element of $\mathrm{Isom}(D_\infty, \Mu)$, since they are conjugate. 

\vspace{.15in}
\noindent{\bf Remark 1.}  The Seifert fibrations $({\ast}{:}{\times})$ and  $(2{\bar\ast}2{:}2)$, with IT numbers 9 and 15, respectively, 
in Table 1 of \cite{C-T} were replaced by two different affinely equivalent Seifert fibrations in Table 1 of \cite{R-T}. 
The Seifert fibrations $({\ast}{:}{\times})$ and  $(2{\bar\ast}2{:}2)$ have orthogonally dual co-Seifert fibers of type $\times\times$. 
The structure group actions for these fibrations are (c-ref., 2-rot.) and  (v-ref., ref.),  (c-ref., 2-rot.) respectively. 
The action (c-ref., 2-rot.) corresponds to the pair \{c-ref., c-ref.\}, 
and the action (v-ref., ref.),  (c-ref., 2-rot.) corresponds to the pair \{v-ref., 2-rot.\}. 
Hence, these co-Seifert fibrations are affinely equivalent to the co-Seifert fibrations described in rows 2 and 5 of Table \ref{T14}  respectively, by Theorems 7 and 9 of \cite{R-T-Co} respectively.

\section{Generic Fiber $\ast\ast$ (annulus) with IT number 3} 

Let $\Mu = \langle e_1+I, e_2+I, A\rangle$ where $A = \mathrm{diag}(1,-1)$. 
Then $\Mu$ is a 2-space group with IT number 3. 
The Conway notation for $\Mu$ is $\ast\ast$ and the IT notation is $pm$. 
The flat orbifold $E^2/\Mu$ is an annulus. 
We have that $Z(\Mu) = \langle e_1+I \rangle$. 

\begin{lemma}\label{L17}  
If $\Mu = \langle e_1+I, e_2 + I, A \rangle$, with $A = \mathrm{diag}(1,-1)$, then 
$$N_A(\Mu) = \{b+B: 2b_2 \in \integers \ \, \hbox{and}\ \, B \in \langle -I, A\rangle\}.$$
\end{lemma}
\begin{proof}
Observe that $b+B \in N_A(\Mu)$ if and only if $B\in N_A(\langle e_1+I, e_2+I\rangle)$ 
and $(b+B)A(b+B)^{-1}\in \Mu$.  
Now $B \in N_A(\langle e_1+I, e_2+I\rangle)$ if and only if $B \in \mathrm{GL}(2,\integers)$.
As $(b+B)A(b+B)^{-1} = b-BAB^{-1}b+BAB^{-1}$, we have that $(b+B)A(b+B)^{-1}\in \Mu$ 
if and only if $BAB^{-1} = A$ and $b-Ab\in \integers^2$. 
Hence $b+B \in N_A(\Mu)$ if and only if $B \in \langle -I, A\rangle$ and $2b_2 \in \integers$. 
\end{proof}

A fundamental polygon for the action of $\Mu$ on $E^2$ is the rectangle with vertices 
$v_1 = (0,0), v_2 = (1,0), v_3 = (1,1/2)$ and $v_4= (0,1/2)$. 
The reflection $A$ fixes the side $[v_1, v_2]$ pointwise, the reflection $e_2+A$ fixes the side $[v_4, v_3]$ pointwise, 
and the horizontal translation $e_1+I$ maps the side $[v_1, v_4]$ to the side $[v_2, v_3]$. 
The central geodesic of the annulus $E^2/\Mu$, which we call the {\it central circle}, represented by the horizontal 
line segment $[(0,1/4),(1,1/4)]$ is invariant under any symmetry of the annulus. 
The symmetry group of $E^2/\Mu$ is the direct product of the subgroup of order 2, generated by 
the reflection in the central circle, and the subgroup consisting of the horizontal rotations and the reflections 
in a pair of antipodal vertical line segments.  The latter subgroup restricts to the symmetry group of the central circle.  

There are five conjugacy classes of symmetries of order 2, 
the class of the reflection c-ref.\  = $(e_2/2+I)_\star$ in the central circle, 
the class of the horizontal halfturn 2-sym.\ = $(e_1/2+I)_\star$, 
the class of the halfturn glide-reflection g-ref.\ = $(e_1/2+e_2/2+I)_\star$ in the central circle, 
the class of a reflection v-ref.\ = $(-I)_\star$ in a pair of antipodal vertical line segments of the annulus, 
and the class of the halfturn 2-rot.\  = $(e_2/2-I)_\star$ around a pair of antipodal points on the central circle. 
Define v-ref.$' = (e_1/2-I)_\star$ and  2-rot.$' = (e_1/2+e_2/2-I)_\star$.  

There are five conjugacy classes of dihedral symmetry groups of order 4, 
the classes of the groups $\{$idt., c-ref., 2-sym., g-ref.$\}$, $\{$idt.,  c-ref., v-ref., 2-rot.$\}$, $\{$idt., 2-sym., v-ref., v-ref.$'$$\}$, $\{$idt.,  2-sym., 2-rot., 2-rot.$'$$\}$, and 
$\{$idt.,  g-ref. v-ref., 2-rot.$'$$\}$. 

\begin{lemma}\label{L18} 
If $\Mu = \langle e_1+I, e_2 + I, A \rangle$, with $A = \mathrm{diag}(1,-1)$, 
then the Lie group $\mathrm{Sym}(\Mu)$ is isomorphic to $C_2\times\mathrm{O}(2)$, 
and $\mathrm{Sym}(\Mu) = \mathrm{Aff}(\Mu)$, 
and $\Omega:  \mathrm{Aff}(\Mu) \to \mathrm{Out}(\Mu)$ maps 
the subgroup {\rm \{idt., c-ref., v-ref., 2-rot.\}}  isomorphically onto $\mathrm{Out}(\Mu)$.  
\end{lemma}
\begin{proof} Let $\mathrm{O}$ be the central circle of the annulus, 
and let $\Lambda$ be the group of all symmetries of the annulus that leave each horizontal geodesic invariant. 
Restriction induces an isomorphism from $\Lambda$ to $\mathrm{Isom}(\mathrm{O})$ by Lemmas 1 of 
\cite{R-T-I} and \ref{L17}. 
We have that  $\mathrm{Sym}(\Mu) = \langle$c-ref.$\rangle \times \Lambda$, 
since every symmetry of the annulus leaves $\mathrm{O}$ invariant,   
and c-ref.\ commutes with every symmetry in $\Lambda$. 
Therefore $\mathrm{Sym}(\Mu)$ is isomorphic to $C_2\times \mathrm{O}(2)$. 
We have that $\mathrm{Sym}(\Mu) = \mathrm{Aff}(\Mu)$ by Lemmas 1 and 7 of \cite{R-T-I} and \ref{L17}. 
The epimorphism $\Omega:  \mathrm{Aff}(\Mu) \to \mathrm{Out}(\Mu)$ maps  {\rm \{idt., c-ref., v-ref., 2-rot.\}} isomorphically onto 
$\mathrm{Out}(\Mu)$ by Theorem 3 of \cite{R-T-I}, since v-ref.\ acts as a reflection of $\mathrm{O}$. 
\end{proof}

\begin{table}  
\begin{tabular}{rlllll}
no. & fibers & grp. & quotients &  structure group action & classifying pair \\
\hline 
    6 & $(\ast\ast, \mathrm{O})$ & $C_1$ & $(\ast\ast, \mathrm{O})$ & (idt., idt.) &  \{idt., idt.\} \\
    8 & $(\ast\ast, \mathrm{O})$ & $C_2$ & $(\ast\ast, \mathrm{O})$ & (c-ref., 2-rot.) & \{c-ref., c-ref.\}  \\
  10 & $(\ast\ast, \mathrm{O})$ & $C_2$ & $(\ast 2222, \mathrm{I})$ & (v-ref., ref.) &  \{v-ref., v-ref.$\}_\ast$ \\
  11 & $(\ast\ast, \mathrm{O})$& $C_2$ & $(22\ast, \mathrm{I})$ & (2-rot., ref.) & \{2-rot., 2-rot.$\}_\ast$\\
  12 & $(\ast\ast, \mathrm{O})$ & $D_2$ & $(\ast 2222, \mathrm{I})$ & (v-ref., ref.), (c-ref., 2-rot.) & \{v-ref.,  2-rot.$\}_\ast$ \\
  25 & $(\ast\ast, \mathrm{I})$ & $C_1$ & $(\ast\ast, \mathrm{I})$ & (idt., idt.) &  \{idt., idt.\} \\
  26 & $(\ast\ast, \mathrm{O})$ & $C_2$ & $(\ast\ast, \mathrm{I})$ & (2-sym., ref.) & \{2-sym., 2-sym.\} \\
  26 & $(\ast\ast, \mathrm{O})$  & $C_2$ & $(\ast 2222, \mathrm{O})$ & (v-ref., 2-rot.) &  \{v-ref., v-ref.\} \\
  28 & $(\ast\ast, \mathrm{O})$ & $C_2$ & $(\ast\ast, \mathrm{I})$ & (c-ref., ref.) & \{c-ref., c-ref.\}  \\
  31 & $(\ast\ast, \mathrm{O})$ & $C_2$ & $(\ast\times, \mathrm{I})$ & (g-ref., ref.) & \{g-ref., g-ref.\} \\
  31 & $(\ast\ast, \mathrm{O})$ & $C_2$ & $(22\ast, \mathrm{O})$ & (2-rot., 2-rot.) & \{2-rot., 2-rot.\} \\
  35 & $(\ast\ast, \mathrm{I})$ & $D_1$ & $(\ast\ast, \mathrm{I})$ & (c-ref., ref.) & \{idt., c-ref.\} \\
  36 & $(\ast\ast, \mathrm{O})$  & $D_2$ & $(\ast\ast, \mathrm{I})$ & (2-sym., ref.), (c-ref., 2-rot.) & \{2-sym., g-ref.\} \\
  38 & $(\ast\ast, \mathrm{I})$ & $D_1$ & $(\ast\ast, \mathrm{I})$ & (2-sym., ref.) &  \{idt., 2-sym.\} \\
  40 & $(\ast\ast, \mathrm{O})$   & $D_2$ & $(\ast\ast, \mathrm{I})$ & (c-ref., ref.), (2-sym., 2-rot.) &  \{c-ref.,  g-ref.\} \\
  44 & $(\ast\ast, \mathrm{I})$ & $D_1$ & $(\ast\times, \mathrm{I})$ & (g-ref., ref.) &  \{idt., g-ref.\} \\
  46 & $(\ast\ast, \mathrm{O})$ & $D_2$ & $(\ast\ast, \mathrm{I})$ & (c-ref., ref.), (g-ref., 2-rot.) & \{c-ref., 2-sym.\} \\
  51 & $(\ast\ast, \mathrm{I})$   & $D_1$ & $(\ast 2222, \mathrm{I})$ & (v-ref., ref.) &  \{idt., v-ref.\} \\
  53 & $(\ast\ast, \mathrm{O})$ & $D_2$ & $(\ast 2222, \mathrm{I})$ & (v-ref., ref.), (2-rot., 2-rot.) & \{v-ref., c-ref.\}\\
  55 & $(\ast\ast, \mathrm{O})$ & $D_2$ & $(\ast 2222, \mathrm{I})$ & (2-sym., ref.), (v-ref.$'$, 2-rot.) & \{2-sym., v-ref.\} \\
  57 & $(\ast\ast, \mathrm{O})$  & $D_2$ & $(\ast 2222, \mathrm{I})$ & (c-ref., ref.), (v-ref., 2-rot.) & \{c-ref.,  2-rot.\} \\
  58 & $(\ast\ast, \mathrm{O})$ & $D_2$ & $(2{\ast}22, \mathrm{I})$ & (g-ref., ref.), (2-rot.$'$, 2-rot.) & \{g-ref., v-ref.\} \\
  59 & $(\ast\ast, \mathrm{I})$   & $D_1$ & $(22\ast, \mathrm{I})$ & (2-rot., ref.) &  \{idt., 2-rot.\} \\
  62 & $(\ast\ast, \mathrm{O})$ & $D_2$ & $(22\ast, \mathrm{I})$ & (2-sym., ref.), (2-rot.$'$, 2-rot.) & \{2-sym., 2-rot.\} \\
  62 & $(\ast\ast, \mathrm{O})$ & $D_2$ & $(2{\ast}22, \mathrm{I})$ & (g-ref., ref.), (v-ref., 2-rot.) &  \{g-ref., 2-rot.$'$\}

\end{tabular}

\medskip
\caption{The classification of the co-Seifert fibrations of 3-space groups 
whose generic fiber is of type $\ast\ast$ with IT number 3}\label{T15}
\end{table}

The derivation of Table \ref{T15} for fiber of type $\ast\ast$ is similar to the derivation of Table \ref{T14} for fiber of type $\times\times$.

\section{Generic Fiber $2222$ (pillow) with IT number 2} 

Let $\Mu = \langle e_1+I, e_2+I, -I\rangle$ where $e_1, e_2$ are the standard basis vectors of $E^2$. 
Then $\Mu$ is a 2-space group with IT number 2. 
The Conway notation for $\Mu$ is $2222$ and the IT notation is $p2$. 
The flat $E^2/\Mu$ is a {\it square pillow} $\Box$. 
A flat 2-orbifold affinely equivalent to $\Box$ is called a {\it pillow}. 
A pillow is orientable and has four $180^\circ$ cone points. 

\begin{lemma}\label{L19} 
If $\Mu = \langle e_1+I, e_2+I, -I\rangle$, then 
$\Omega: \mathrm{Aff}(\Mu) \to \mathrm{Out}(\Mu)$ is an isomorphism. 
\end{lemma}
\begin{proof}
We have that $Z(\Mu) =\{I\}$ by Lemma 5 of \cite{R-T-I}. 
Hence $\Omega: \mathrm{Aff}(\Mu) \to \mathrm{Out}(\Mu)$ is an isomorphism by Theorems 1 and 3 of \cite{R-T-I}. 
\end{proof}

\begin{lemma}\label{L20} 
If $\Mu = \langle e_1+I, e_2+I, -I\rangle$, then 
$$N_A(\Mu) = \left\{\frac{m}{2}e_1+ \frac{n}{2}e_2  + A :  m, n \in \integers\ \hbox{and}\ A \in \mathrm{GL}(2,\integers)\right\}.$$
\end{lemma}
\begin{proof}
Observe that $a+ A \in N_A(\Mu)$ if and only if $A \in N_A(\langle e_1+I, e_2+I\rangle)$ and $(a+I)(-I)(a+I)^{-1} \in \Mu$. 
Now $A \in N_A(\langle e_1+I, e_2+I\rangle)$ if and only if $A \in  \mathrm{GL}(2,\integers)$. 
As $(a+I)(-I)(a+I)^{-1} = 2a - I$, we have that 
$(a+I)(-I)(a+I)^{-1} \in \Mu$ if and only if  $a = \frac{m}{2}e_1+ \frac{n}{2}e_2$ for some $m, n \in \integers$. 
\end{proof}

\begin{lemma}\label{L21} 
Let $\Mu = \langle e_1+I, e_2+I, -I\rangle$. 
The group $\mathrm{Aff}(\Mu)$ has a normal dihedral subgroup $\Kappa$ of order 4 generated by 
$(e_1/2+ I)_\star$ and $(e_2/2 +I)_\star$. 
The map $\eta: \mathrm{Aff}(\Mu) \to \mathrm{PGL}(2,\integers)$, defined by 
$\eta((a+A)_\star) = \pm A$ for each $A \in \mathrm{GL}(2,\integers)$, 
is an epimorphism with kernel $\Kappa$. 
The map $\sigma:  \mathrm{PGL}(2,\integers) \to \mathrm{Aff}(\Mu)$, defined by $\sigma(\pm A) = A_\star$, 
is a monomorphism,  and $\sigma$ is a right inverse of $\eta$. 
\end{lemma}
\begin{proof}
We have that $Z(\Mu)= \{I\}$ by Lemma 5 of \cite{R-T-I},  
and $\Omega: \mathrm{Aff}(\Mu) \to \mathrm{Out}(\Mu)$ is 
an isomorphism by Theorems 1 and 3 of \cite{R-T-I}.  
Hence $\Kappa = \Omega^{-1}(\mathrm{Out}_E^1)$ is a normal subgroup of $\mathrm{Aff}(\Mu)$ by Lemma 9 of \cite{R-T-I}.
Let $\Tau$ be the translation subgroup of $N_A(\Mu)$.  Then $\Kappa = \Phi(\Tau\Mu/\Mu)$ 
by Lemmas 7 and 8 of \cite{R-T-I}. 
Now $\Tau\Mu/\Mu \cong \Tau/\Tau\cap \Mu$ is a dihedral group of order 4 generated by 
$(e_1/2+ I)\Mu$ and $(e_2/2 +I)\Mu$ by Lemma \ref{L20}. 
Hence $\Kappa$ is a dihedral group generated by $(e_1/2+ I)_\star$ and $(e_2/2 +I)_\star$. 

The point group $\Pi_A$ of $N_A(\Mu)$ is $\mathrm{GL}(2,\integers)$ by Lemma \ref{L20}. 
Hence, the map $\eta: \mathrm{Aff}(\Mu) \to \mathrm{PGL}(2,\integers)$, defined by 
$\eta((a+A)_\star) = \pm A$ for each $A \in \mathrm{GL}(2,\integers)\}$, 
is an epimorphism with kernel $\Kappa$ by Lemma 9 of \cite{R-T-I}. 
The map $\sigma:  \mathrm{PGL}(2,\integers) \to \mathrm{Aff}(\Mu)$ 
is a well-defined homomorphism by Lemma 7 of \cite{R-T-I}. 
The map $\sigma$ is a monomorphism, since $\sigma$ is a right inverse of $\eta$.  
\end{proof}

\begin{lemma}\label{L22} 
Let 
$$A = \left(\begin{array}{rr} -1 & 0 \\ 0 & 1 \end{array}\right),\quad  B = \left(\begin{array}{rr} -1 & 1 \\ 0 & 1 \end{array}\right), 
\quad C = \left(\begin{array}{rr} 0 & 1 \\ 1 & 0 \end{array}\right). $$
The group $\mathrm{PGL}(2,\integers)$ is the free product of the dihedral subgroup 
$\langle \pm A,  \pm C\rangle$ of order 4 
and the dihedral subgroup $\langle\pm B,  \pm C \rangle$ of order 6 
amalgamated along the subgroup $\langle \pm C\rangle$ of order 2. 
Every finite subgroup of  $\mathrm{PGL}(2,\integers)$ is conjugate to a subgroup 
of either $\langle \pm A,  \pm C\rangle$ or $\langle\pm B,  \pm C \rangle$. 
\end{lemma}
\begin{proof}
The matrices $A, B, C$ have order 2, and $(AC)^2 = -I$, and $(BC)^3=-I$. 
Hence $\langle \pm A, \pm C\rangle$ is a dihedral group of order 4,  
and $\langle \pm B, \pm C\rangle$ is a dihedral group of order 6. 

The group $\mathrm{PGL}(2,\integers)$ acts effectively by isometries on the upper complex half-space model of hyperbolic 2-space $H^2$ by
$$\pm\left(\begin{array}{cc} a & b \\ c & c\end{array}\right)z = \frac{az+b}{cz+d} \quad \hbox{if}\  
\det\left(\begin{array}{cc} a & b \\ c & c\end{array}\right) = 1,$$
and 
$$\pm\left(\begin{array}{cc} a & b \\ c & c\end{array}\right)z = \frac{a\overline{z}+b}{c\overline{z}+d} \quad \hbox{if}\  
\det\left(\begin{array}{cc} a & b \\ c & c\end{array}\right) = -1,$$
It is well know that a fundamental domain for the action of  $\mathrm{PGL}(2,\integers)$ on $H^2$ is the hyperbolic triangle $T$ with vertices $i, \frac{1}{2}+\frac{\sqrt{3}}{2}i, \infty$ and corresponding angles $\pi/2, \pi/3, 0$. 
Moreover $\mathrm{PGL}(2,\integers)$ is a triangle reflection group with respect to the reflections  
$R_1(z) = -\overline{z}, R_2(z) = 1/\overline{z}, R_3(z) = 1-\overline{z}$ in the sides of $T$. 
Note that $R_1$ is the reflection in the $y$-axis, $R_2$ is the inversion in the circle $|z| = 1$, 
and $R_3$ is the reflection in the line $x = \frac{1}{2}$. 
Moreover
$$(\pm A)z = R_1(z), \quad (\pm C)z = R_2(z), \quad (\pm B)z = R_3(z).$$

By Poincar{\'e}'s fundamental polygon theorem (Theorem 13.5.3 \cite{R}), 
the group $\mathrm{PGL}(2,\integers)$ has the presentation 
$$\langle R_1, R_2, R_3; R_1^2, R_2^2, R_3^2, (R_1R_2)^2, (R_2R_3)^3\rangle,$$
which is equivalent to the presentation
$$\langle R_1, R_2, R_3, R_4; R_1^2, R_2^2, (R_1R_2)^2, R_2 = R_4, R_3^2, R_4^2, (R_4R_3)^3\rangle.$$
Therefore $\mathrm{PGL}(2,\integers)$ is the free product with amalgamation 
of the dihedral group $\langle R_1, R_2\rangle$ of order 4 and the dihedral group 
$\langle R_2, R_3\rangle$ of order 6 amalgamated along the subgroup $\langle R_2\rangle$ of order 2.

By the torsion theorem for amalgamated products of groups, every finite subgroup of 
$\mathrm{GL}(2,\integers)$ is conjugate to a subgroup of $\pm\langle A, \pm C\rangle$ or $\langle \pm B, \pm C\rangle$. 
\end{proof}

\begin{lemma}\label{L23} 
Let $\Mu = \langle e_1+I, e_2+I, -I\rangle$, let $\Kappa = \langle (e_1/2+ I)_\star, (e_2/2 +I)_\star \rangle$, and  
let $A, B, C$ be defined as in Lemma \ref{L22}. 
The group $\mathrm{Aff}(\Mu)$  is the free product of the subgroup $\langle \Kappa, A_\star, C_\star\rangle$ of order 16
and the subgroup $\langle \Kappa, B_\star, C_\star\rangle$ of order 24 amalgamated 
along the subgroup $\langle \Kappa, C_\star\rangle$ of order 8. 
Every finite subgroup of $\mathrm{Aff}(\Mu)$ is conjugate to a subgroup of either
$\langle \Kappa, A_\star, C_\star\rangle$ or $\langle \Kappa, B_\star, C_\star\rangle$. 
\end{lemma}
\begin{proof} 
The map $\eta: \mathrm{Aff}(\Mu) \to \mathrm{PGL}(2,\integers)$, defined by $\eta((a+A)_\star) = \pm A$, 
is an epimorphism with kernel $\Kappa$ by Lemma \ref{L21}.  
The group $\mathrm{Aff}(\Mu)$ acts on the Bass-Serre tree of the amalgamated product decomposition of $\mathrm{PGL}(2,\integers)$ via $\eta$. 
Hence $\mathrm{Aff}(\Mu)$ is the amalgamated product of the groups
$\eta^{-1}(\langle \pm A, \pm C\rangle) = \langle \Kappa, A_\star, C_\star\rangle$ 
and $\eta^{-1}(\langle \pm B, \pm C\rangle) = \langle \Kappa, B_\star, C_\star \rangle$
along the subgroup $\eta^{-1}(\langle \pm C\rangle) = \langle \Kappa, C_\star\rangle$. 
\end{proof}

Let $\Delta = \langle e_1+I, e_1/2 + \sqrt{3}e_2/2+I, -I\rangle$. 
Then $\Delta$ is a 2-space group, and the flat orbifold $E^2/\Delta$ is a {\it tetrahedral pillow} $\triangle$. 

\begin{lemma}\label{L24} 
Let $\Mu = \langle e_1+I, e_2+I, -I\rangle$, and let $\Delta = \langle e_1+I, e_1/2 + \sqrt{3}e_2/2+I, -I\rangle$. 
Let $\Kappa = \langle (e_1/2+ I)_\star, (e_2/2 +I)_\star \rangle$, and  
let 
$$A = \left(\begin{array}{rr} -1 & 0 \\ 0 & 1 \end{array}\right),\ \  B = \left(\begin{array}{rr} -1 & 1 \\ 0 & 1 \end{array}\right), 
\ \ C = \left(\begin{array}{rr} 0 & 1 \\ 1 & 0 \end{array}\right), \ \ 
D  = \left(\begin{array}{rr} 1 & -1/2 \\ 0 & \sqrt{3}/2 \end{array}\right). $$
Then $\mathrm{Sym}(\Mu) = \langle \Kappa, A_\star, C_\star\rangle$,  and $D\Mu D^{-1}=\Delta$, 
and $D_\sharp: \mathrm{Aff}(\Mu) \to \mathrm{Aff}(\Delta)$ is an isomorphism, 
and $\mathrm{Sym}(\Delta) = D_\sharp(\langle \Kappa, B_\star, C_\star\rangle)$. 
\end{lemma}
\begin{proof}
We have that $Z(\Mu) = \{I\}$ by Lemma 5 of \cite{R-T-I}, and so $\mathrm{Sym}(\Mu)$ is finite by Corollary 2 of \cite{R-T-I}. 
Now $\mathrm{Sym}(\Mu)$ is a finite subgroup of $\mathrm{Aff}(\Mu)$ that contains $\langle \Kappa, A_\star, C_\star\rangle$. 
Hence $\mathrm{Sym}(\Mu)= \langle \Kappa, A_\star, C_\star\rangle$, since $\langle \Kappa, A_\star, C_\star\rangle$ 
is a maximal finite subgroup of $\mathrm{Aff}(\Mu)$ by Lemma \ref{L23}. 

Clearly $D\Mu D^{-1}=\Delta$, and so $D_\sharp: \mathrm{Aff}(\Mu) \to \mathrm{Aff}(\Delta)$ is an isomorphism. 
Now $D_\sharp(\langle \Kappa, B_\star, C_\star\rangle)$ is a subgroup of $\mathrm{Sym}(\Delta)$, 
since $DBD^{-1}, DCD^{-1} \in \mathrm{O}(2)$.  
Hence $D_\sharp^{-1}(\mathrm{Sym}(\Delta))$ is a finite subgroup of $\mathrm{Aff}(\Mu)$ 
that contains $\langle \Kappa, B_\star, C_\star\rangle$. 
Therefore  $D_\sharp^{-1}(\mathrm{Sym}(\Delta)) = \langle \Kappa, B_\star, C_\star\rangle$, 
since $\langle \Kappa, B_\star, C_\star\rangle$ is a maximal finite subgroup of $\mathrm{Aff}(\Mu)$ by Lemma \ref{L23}. 
Hence $\mathrm{Sym}(\Delta) = D_\sharp(\langle \Kappa, B_\star, C_\star\rangle)$. 
\end{proof}

A square pillow is formed by identifying the boundaries of two congruent squares.  
The symmetry group of a square pillow is the direct product of the subgroup of order 2 generated 
by the central reflection between the two squares, 
and the subgroup of order 8 corresponding to the symmetry group of the two squares. 
A fundamental domain for the square pillow $\Box = E^2/\Mu$ is the rectangle with vertices $(0,0), (1/2,0), (1/2,1), (0, 1)$. 
This rectangle is subdivided into two congruent squares that correspond to the two sides of the pillow $\Box$. 

There are seven conjugacy classes of symmetries of order 2 of $\Box$ represented by 

\begin{enumerate}
\item  the {\it central halfturn} c-rot.\ = $(e_1/2+e_2/2+I)_\star$ about the centers of the squares, 

\item  the {\it midline halfturn} m-rot.\ = $(e_1/2+ I)_\star$ about the midpoints of opposite sides of the squares,  

\item  the {\it central reflection} c-ref.\ = $A_\star$ between the two squares, 

\item the {\it midline reflection} m-ref.\ = $(e_1/2+A)_\star$,  

\item the {\it antipodal map} 2-sym.\ = $(e_1/2+e_2/2 + A)_\star$, 

\item the {\it diagonal reflection} d-ref.\ = $C_\star$, and 

\item the {\it diagonal halfturn} d-rot.\ = $(AC)_\star$. 
\end{enumerate}

There are two conjugacy classes of symmetries of order 4 of $\Box$, 
the class of the $90^\circ$ rotation 4-rot.\ $= (e_1/2+ AC)_\star$ about the centers of the squares, 
and the class of 4-sym.\ = $(e_1/2+ C)_\star$. 

There are nine conjugacy classes of dihedral symmetry groups of order 4 of $\Box$, 
the classes of the groups $\Kappa = 
\{$idt., c-rot., m-rot., m-rot.$'$$\}$,  
$\{$idt., m-ref.,  2-sym., m-rot.$'$$\}$, 
$\{$idt.,  c-rot., d-ref., d-ref.$'$$\}$, 
$\{$idt., c-ref., m-ref., m-rot.$\}$, 
$\{$idt.,  c-rot. m-ref., m-ref.$'$$\}$, 
$\{$idt., c-ref., c-rot., 2-sym.$\}$, 
$\{$idt.,  c-rot., d-rot., d-rot.$'$$\}$, 
$\{$idt., c-ref., d-ref., d-rot.$\}$, and 
$\{$idt.,  2-sym., d-ref., d-rot.$'$$\}$. 

There are four conjugacy classes of dihedral symmetry groups of order 8 of $\Box$, 
the classes of the groups 
$\langle$m-rot., d-rot.$\rangle$, 
$\langle$m-ref., d-ref.$\rangle$, 
$\langle$m-rot., d-ref.$\rangle$, and 
$\langle$m-ref., d-rot.$\rangle$.  
Moreover 4-rot.\ = (m-rot.)(d-rot.) = (m-ref.)(d-ref.) and 4-sym.\ = (m-rot.)(d-ref.) = (m-ref.)(d-rot.). 

A tetrahedral pillow is realized by the boundary of a regular tetrahedron. 
The symmetry group of a tetrahedral pillow corresponds to the symmetric group on its four cone points (vertices). 
A fundamental domain for the tetrahedral pillow $\triangle = E^2/\Delta$ is 
the equilateral triangle with vertices $(0,0), (1,0), (1/2, \sqrt{3}/2)$. 
This triangle is subdivided into four congruent equilateral triangles that correspond 
to the four faces of a regular tetrahedron. 

There are two conjugacy classes of symmetries of order 2 of $\triangle$, 
the class of the {\it halfturn} 2-rot.\ = $(e_1/2+I)_\star$, 
with axis joining the midpoints of a pair of opposite edges of the tetrahedron, 
corresponding to the product of two disjoint transpositions of vertices, and 
the class of the {\it reflection} ref.\ = $(DBD^{-1})_\star = A_\star$ corresponding to a transposition of vertices.

There is one conjugacy class of symmetries of order 3 of $\triangle$, 
the class of the $120^\circ$ rotation 3-rot.\ = $(DBCD^{-1})_\star$ corresponding to a 3-cycle.   
There is one conjugacy class of elements of order 4 of $\triangle$, 
the class of 4-cyc.\ = $(e_1/2+ DCD^{-1})_\star$, corresponding to a 4-cycle.

There are two conjugacy classes of dihedral symmetry groups of order 4 of $\triangle$, 
the class of the group of halfturns, 
and the class of the group generated by two perpendicular reflections.

There is one conjugacy class of dihedral symmetry groups of order 6 of $\triangle$, 
the class of the stabilizer of a vertex (and the opposite face) of $\triangle$. 
There is one conjugacy class of dihedral symmetry groups of order 8 of $\triangle$, 
since all Sylow 2-subgroups of the symmetry group of $\triangle$ are conjugate.  

The square pillow $\Box$ is affinely equivalent to the tetrahedral pillow $\triangle$ by Lemma \ref{L24}.  
The symmetries m-rot.\ and 4-sym.\ of $\Box$ are conjugate by $D_\star$ to the symmetries 
2-rot.\ and 4-cyc.\ of $\triangle$ respectively. 
The affinity 3-aff.\ $ = (BC)_\star$ of $\Box$ is conjugate by $D_\star$ to the symmetry 3-rot.\ of $\triangle$. 
The group $\{$idt., c-rot, m-rot., m-rot.$'$$\}$ of symmetries of $\Box$ is conjugate by $D_\star$ to the group of halfturns 
of $\triangle$. 
The group $\{$idt., c-rot, d-ref., d-ref.$'$$\}$ of symmetries of $\Box$ is conjugate by $D_\star$ to a group of symmetries 
of $\triangle$ generated by two perpendicular reflections. 

Define the affinity 2-aff.\  of $\Box$ by 2-aff.\ = $B_\star$. 
Define the reflection ref.$'$ of $\triangle$ by ref.$'$ = $(DCD^{-1})_\star$. 
The group $\langle$2-aff., d-ref.$\rangle$ of affinities of $\Box$ is conjugate by $D_\star$ 
to the dihedral group $\langle$ref., ref.$'\rangle$ of symmetries of $\triangle$ of order 6. 
The group $\langle$m-rot., d-ref.$\rangle$ of symmetries of $\Box$ is conjugate by $D_\star$ 
to the dihedral group of symmetries $\langle$2-rot., ref.$'\rangle$ of $\triangle$ of order 8.

Every affinity of $\Box$ of finite order has order 1, 2, 3, or 4 by Lemma \ref{L23}. 
There are six conjugacy classes of affinities of $\Box$ of order 2 represented 
by m-rot., c-ref., m-ref., 2-sym., d-ref., and d-rot. 
Note that c-rot.\,and m-rot.\,are conjugate in $\mathrm{Aff}(\Mu)$, 
and 2-aff.\,and d-ref.\,are conjuate in $\mathrm{Aff}(\Mu)$, since $B$ and $C$ are 
conjugate in $\mathrm{GL}(2,\integers)$. 
There is one conjugacy classes of affinities of $\Box$ of order 3 represented by 3-aff.  
There are two conjugacy classes of affinities of $\Box$ of order 4 represented by 4-rot.\ and 4-sym. 

We represent $\mathrm{Out}(\Mu)$ by $\mathrm{Aff}(\Mu)$ via the isomorphism $\Omega: \mathrm{Aff}(\Mu) \to \mathrm{Out}(\Mu)$. 
The set $\mathrm{Isom}(C_\infty,\Mu)$ consists of ten elements corresponding to the pairs of inverse elements \{idt., idt.\}, 
\{m-rot., m-rot\}, \{c-ref., c-ref.\}, \{m-ref., m-ref.\}, \{2-sym., 2-sym.\}, \{d-ref., d-ref.\}, \{d-rot., d-rot.\}, \{3-aff., 3-aff.$^{-1}$\}, 
\{4-sym., 4-sym.$^{-1}$\}, \{4-rot., 4-rot.$^{-1}$\} of $\mathrm{Aff}(\Mu)$ by Theorem 7 of \cite{R-T-Co}. 

The set $\mathrm{Isom}(D_\infty,\Mu)$ consists of forty four elements corresponding to the remaining 
classifying pairs of elements of $\mathrm{Aff}(\Mu)$ in Table \ref{T16} by Theorem 8 of \cite{R-T-Co}. 

\begin{table}  
\begin{tabular}{rlllll}
no. & fibers & grp. & quotients & structure group action & classifying pair \\
\hline 
    3 & $(2222, \mathrm{O})$ & $C_1$ & $(2222, \mathrm{O})$ & (idt., idt.) &  \{idt., idt.\} \\
    5 & $(2222, \mathrm{O})$ & $C_2$ & $(2222, \mathrm{O})$ & (m-rot., 2-rot.) & \{m-rot., m-rot.\} \\
  10 & $(2222, \mathrm{I})$ & $C_1$ & $(2222, \mathrm{I})$ & (idt., idt.)  &  \{idt., idt.\} \\
  12 & $(2222, \mathrm{I})$ & $D_1$ & $(2222, \mathrm{I})$ & (m-rot., ref.) & \{idt., m-rot.\} \\
  13 & $(2222, \mathrm{O})$ & $ C_2$ & $(2222, \mathrm{I})$ & (m-rot., ref.) & \{m-rot., m-rot.\} \\
  15 & $(2222, \mathrm{O})$ & $D_2$ & $(2222, \mathrm{I})$ & (c-rot., ref.), (m-rot.$'$, 2-rot.) & \{c-rot., m-rot.\} \\
  16 & $(2222, \mathrm{O})$ & $C_2$ & $(\ast 2222, \mathrm{I})$ & (c-ref., ref.) &  \{c-ref., c-ref.\} \\
  17 & $(2222, \mathrm{O})$  & $C_2$ & $(22\ast, \mathrm{I})$ & (m-ref., ref.) & \{m-ref., m-ref.\} \\
  18 & $(2222, \mathrm{O})$ & $C_2$ & $(22\times, \mathrm{I})$ & (2-sym., ref.) & \{2-sym., 2-sym.\} \\
  20 & $(2222, \mathrm{O})$ & $D_2$ & $(22\ast, \mathrm{I})$ & (2-sym., ref.), (m-rot.$'$, 2-rot.)  & \{2-sym., m-ref.\} \\
  21 & $(2222, \mathrm{O})$ & $C_2$ & $(2{\ast}22, \mathrm{I})$ & (d-ref., ref.) &  \{d-ref., d-ref.\} \\
  21 & $(2222, \mathrm{O})$ & $D_2$ & $(\ast 2222, \mathrm{I})$ & (c-ref., ref.), (m-rot., 2-rot.) & \{c-ref., m-ref.\} \\
  22 & $(2222, \mathrm{O})$ & $D_2$ & $(\ast 2222, \mathrm{I})$ & (d-ref., ref.), (c-rot., 2-rot.) &  \{d-ref., d-ref.$'$\}  \\
  23 & $(2222, \mathrm{O})$ & $D_2$ & $(2{\ast}22, \mathrm{I})$ & (c-ref., ref.), (c-rot., 2-rot.) & \{c-ref., 2-sym.\} \\
  24 & $(2222, \mathrm{O})$ & $D_2$ & $(2{\ast}22, \mathrm{I})$ & (m-ref., ref.), (c-rot., 2-rot.) & \{m-ref., m-ref.$'$\} \\
  27 & $(2222, \mathrm{O})$ & $C_2$ & $(\ast 2222, \mathrm{O})$ & (c-ref., 2-rot.) &  \{c-ref., c-ref.\} \\
  30 & $(2222, \mathrm{O})$ & $C_2$ & $(22\ast, \mathrm{O})$ & (m-ref., 2-rot.) & \{m-ref., m-ref.\} \\
  34 & $(2222, \mathrm{O})$ & $C_2$ & $(22\times, \mathrm{O})$ & (2-sym., 2-rot.) & \{2-sym., 2-sym.\} \\
  37 & $(2222, \mathrm{O})$ & $C_2$ & $(2{\ast}22, \mathrm{O})$ & (d-ref., 2-rot.) &  \{d-ref., d-ref.\} \\
  43 & $(2222, \mathrm{O})$ & $C_4$ & $(22\times, \mathrm{O})$ & (4-sym., 4-rot.) & \{4-sym., 4-sym.$^{-1}$\} \\
  48 & $(2222, \mathrm{O})$ & $D_2$ & $(2{\ast}22, \mathrm{I})$ & (c-rot., ref.), (2-sym., 2-rot.) &  \{c-rot., c-ref.\} \\
  49 & $(2222, \mathrm{I})$ & $D_1$ & $(\ast 2222, \mathrm{I})$ & (c-ref., ref.) & \{idt., c-ref.\} \\
  50 & $(2222, \mathrm{O})$ & $D_2$ & $(\ast 2222, \mathrm{I})$ & (c-ref., ref.), (m-ref., 2-rot.) & \{c-ref., m-rot.\} \\ 
  52 & $(2222, \mathrm{O})$ & $D_2$ & $(22\ast, \mathrm{I})$ & (m-ref., ref.), (2-sym., 2-rot.) &  \{m-ref., m-rot.$'$\} \\
  52 & $(2222, \mathrm{O})$ & $D_2$ & $(2{\ast}22, \mathrm{I})$ & (c-rot., ref.), (m-ref.$'$, 2-rot.) & \{c-rot., m-ref.\} \\
  53 & $(2222, \mathrm{I})$  & $D_1$ & $(22\ast, \mathrm{I})$ & (m-ref., ref.) &  \{idt., m-ref.\} \\
  54 & $(2222, \mathrm{O})$ & $D_2$ & $(\ast 2222, \mathrm{I})$ & (m-ref., ref.), (c-ref., 2-rot.) &  \{m-ref., m-rot.\} \\
  56 & $(2222, \mathrm{O})$ & $D_2$ & $(2{\ast}22, \mathrm{I})$ & (c-rot., ref.), (c-ref., 2-rot.) & \{c-rot., 2-sym.\} \\
  58 & $(2222, \mathrm{I})$ & $D_1$ & $(22\times, \mathrm{I})$ & (2-sym., ref.) & \{idt., 2-sym.\} \\
  60 & $(2222, \mathrm{O})$ & $D_2$ & $(22\ast, \mathrm{I})$ & (2-sym., ref.), (m-ref., 2-rot.) &  \{2-sym., m-rot.$'$\} \\
  66 & $(2222, \mathrm{I})$ & $D_1$ & $(2{\ast}22, \mathrm{I})$ & (d-ref., ref.) &  \{idt., d-ref.\} \\
  68 & $(2222, \mathrm{O})$ & $D_2$ & $(\ast 2222, \mathrm{I})$ & (c-rot., ref.), (d-ref.$'$, 2-rot.) & \{c-rot., d-ref.\} \\
  70 & $(2222, \mathrm{O})$ & $D_4$ & $(2{\ast}22, \mathrm{I})$ & (m-rot., ref.), (4-sym., 4-rot.) & \{m-rot., d-ref.\} \\
  77 & $(2222, \mathrm{O})$ & $C_2$ & $(442, \mathrm{O})$ & (d-rot., 2-rot.) & \{d-rot., d-rot.\} \\
  80 & $(2222, \mathrm{O})$ & $C_4$ & $(442, \mathrm{O})$ & (4-rot., 4-rot.) & \{4-rot., 4-rot.$^{-1}$\} \\
  81 & $(2222, \mathrm{O})$ & $C_2$ & $(442, \mathrm{I})$ & (d-rot., ref.) &  \{d-rot., d-rot.\} \\
  82 & $(2222, \mathrm{O})$ & $D_2$ & $(442, \mathrm{I})$ & (d-rot., ref.), (c-rot., 2-rot.) & \{d-rot., d-rot.$'$\} \\
  84 & $(2222, \mathrm{I})$ & $D_1$ & $(442, \mathrm{I})$ & (d-rot., ref.) & \{idt., d-rot.\} \\
  86 & $(2222, \mathrm{O})$ & $D_2$ & $(442, \mathrm{I})$ & (c-rot., ref.), (d-rot.$'$, 2-rot.) & \{c-rot., d-rot.\} \\
  88 & $(2222, \mathrm{O})$ & $D_4$ & $(442, \mathrm{I})$ & (m-rot., ref.), (4-rot., 4-rot.) & \{m-rot., d-rot.\} \\
  93 & $(2222, \mathrm{O})$ & $D_2$ & $(\ast 442, \mathrm{I})$ & (c-ref., ref.), (d-rot., 2-rot.) & \{c-ref., d-ref.\} \\
  94 & $(2222, \mathrm{O})$ & $D_2$ & $(4{\ast}2, \mathrm{I})$ & (2-sym., ref.), (d-rot.$'$, 2-rot.) & \{2-sym., d-ref.\} \\
  98 & $(2222, \mathrm{O})$ & $D_4$ & $(\ast 442, \mathrm{I})$ & (m-ref., ref.), (4-rot., 4-rot.) &  \{m-ref, d-ref.\} \\
112 & $(2222, \mathrm{O})$ & $D_2$ & $(\ast 442, \mathrm{I})$ & (c-ref., ref.), (d-ref., 2-rot.) & \{c-ref., d-rot.\} \\
114 & $(2222, \mathrm{O})$ & $D_2$ & $(4{\ast}2, \mathrm{I})$ & (2-sym., ref.), (d-ref., 2-rot.) & \{2-sym, d-rot.$'$\} \\
116 & $(2222, \mathrm{O})$ & $D_2$ & $(\ast 442, \mathrm{I})$ & (d-ref., ref.), (c-ref., 2-rot.) & \{d-ref., d-rot.\} \\
118 & $(2222, \mathrm{O})$ & $D_2$ & $(4{\ast}2, \mathrm{I})$ & (d-ref., ref.), (2-sym., 2-rot.) & \{d-ref., d-rot.$'$\} \\
122 & $(2222, \mathrm{O})$ & $D_4$ & $(4{\ast}2, \mathrm{I})$ & (m-ref., ref.), (4-sym., 4-rot.) &  \{m-ref., d-rot.\} \\
171 & $(2222, \mathrm{O})$ & $C_3$ & $(632, \mathrm{O})$ & (3-rot., 3-rot.) & \{3-aff., 3-aff.$^{-1}$\} \\
180 & $(2222, \mathrm{O})$ & $D_3$ & $(\ast632, \mathrm{I})$ & (ref., ref.), (3-rot., 3-rot.) & \{2-aff., d-ref.\}
\end{tabular}

\medskip
\caption{The classification of the co-Seifert fibrations of 3-space groups 
whose generic fiber is of type $2222$ with IT number 2}\label{T16}
\end{table}

\section{Generic Fiber $\circ$ (torus) with IT number 1} 

Let $\Mu = \langle e_1+I, e_2+I\rangle$ where $e_1, e_2$ are the standard basis vectors of $E^2$. 
Then $\Mu$ is a 2-space group with IT number 1. The Conway notation for $\Mu$ is $\circ$ and the IT notation is $p1$. 
The flat orbifold $E^2/\Mu$ is a {\it square torus} $\Box$.  
A flat 2-orbifold affinely equivalent to $\Box$ is called a {\it torus}. 
A torus is orientable. 

\begin{lemma}\label{L25}  
If $\Mu = \langle e_1+I, e_2+I\rangle$, then 
$\mathrm{Out}(\Mu)= \mathrm{Aut}(\Mu) = \mathrm{GL}(2,\integers)$, and  
$$N_A(\Mu) = \left\{a  + A :  a \in E^2\ \hbox{and}\ A \in \mathrm{GL}(2,\integers)\right\}.$$
\end{lemma}
\begin{proof}
We have that $\mathrm{Out}(\Mu)= \mathrm{Aut}(\Mu)$, since $\Mu$ is abelian, 
and $\mathrm{Aut}(\Mu) = \mathrm{GL}(2,\integers)$, since every automorphism of $\Mu$ 
extends to a unique linear automorphism of $E^2$ corresponding to an element of $\mathrm{GL}(2,\integers)$. 

Observe that $a+ A \in N_A(\Mu)$ if and only if $A \in N_A(\langle e_1+I, e_2+I\rangle)$. 
Hence $a+ A \in N_A(\Mu)$ if and only if  $A \in \mathrm{GL}(2, \integers)$. 
\end{proof}

\begin{lemma}\label{L26} 
Let $\Mu = \langle e_1+I, e_2+I\rangle$. 
The group $\mathrm{Aff}(\Mu)$ has a normal, infinite, abelian subgroup $\Kappa = \{(a+I)_\star: a \in E^2\}$. 
The map $\eta: \mathrm{Aff}(\Mu) \to \mathrm{GL}(2,\integers)$, defined by 
$\eta((a+A)_\star) = A$ for each $A \in \mathrm{GL}(2,\integers)$, 
is an epimorphism with kernel $\Kappa$. 
The map $\sigma:  \mathrm{GL}(2,\integers) \to \mathrm{Aff}(\Mu)$, defined by $\sigma(A) = A_\star$, 
is a monomorphism,  and $\sigma$ is a right inverse of $\eta$. 
\end{lemma}
\begin{proof}
The map $\Phi: N_A(\Mu) \to \mathrm{Aff}(\Mu)$, defined by $\Phi(a+A) = (a+A)_\star$,  
is an epimorphism with kernel $\Mu$ by Lemma 7 of \cite{R-T-I}. 
Let $\Tau$ be the translation subgroup of $N_A(\Mu)$. 
Then $\Phi$ maps the normal subgroup $\Tau/\Mu$ of $N_A(\Mu)/\Mu$ 
onto the normal subgroup $\Kappa$ of $\mathrm{Aff}(\Mu)$. 
We have that 
$$(N_A(\Mu)/\Mu)/(\Tau/\Mu)= N_A(\Mu)/\Tau = \mathrm{GL}(2,\integers).$$
Hence $\eta: \mathrm{Aff}(\Mu) \to \mathrm{GL}(2,\integers)$ is an epimorphism with kernel $\Kappa$. 
The map $\sigma:  \mathrm{GL}(2,\integers) \to \mathrm{Aff}(\Mu)$ 
is a well-defined homomorphism by Lemma 7 of \cite{R-T-I}. 
The map $\sigma$ is a monomorphism, since $\sigma$ is a right inverse of $\eta$.  
\end{proof}

\begin{lemma}\label{L27} 
Let 
$$A = \left(\begin{array}{rr} -1 & 0 \\ 0 & 1 \end{array}\right),\quad  B = \left(\begin{array}{rr} -1 & 1 \\ 0 & 1 \end{array}\right), 
\quad C = \left(\begin{array}{rr} 0 & 1 \\ 1 & 0 \end{array}\right). $$
The group $\mathrm{GL}(2,\integers)$ is the free product of the dihedral subgroup $\langle A,  C\rangle$ of order 8 
and the dihedral subgroup $\langle B,  C \rangle$ of order 12 amalgamated 
along the dihedral subgroup $\langle -I, C \rangle$ of order 4. 
Every finite subgroup of  $\mathrm{GL}(2,\integers)$ is conjugate to a subgroup 
of either $\langle A,  C\rangle$ or $\langle B,  C \rangle$. 
\end{lemma}
\begin{proof}
Let $\pi: \mathrm{GL}(2,\integers) \to \mathrm{PGL}(2,\integers)$ be the natural projection defined by $\pi(A) = \pm A$.  
Then $\mathrm{GL}(2,\integers)$ acts on the Bass-Serre tree of the amalgamated product decomposition of $\mathrm{PGL}(2,\integers)$, given in Lemma \ref{L22}, via $\pi$. 
Hence $\mathrm{GL}(2,\integers)$ is the amalgamated product of the groups
$\pi^{-1}(\langle \pm A, \pm C\rangle) = \langle A, C\rangle$ 
and $\pi^{-1}(\langle \pm B, \pm C\rangle) = \langle B, C \rangle$
along the subgroup $\pi^{-1}(\langle \pm C\rangle) = \langle -I, C \rangle$, since $\mathrm{ker}(\pi) =\{\pm I\}$, and $(AC)^2 = -I$ and $(BC)^3 = -I$.  

By the torsion theorem for amalgamated products of groups, every finite subgroup of $\mathrm{GL}(2,\integers)$ 
is conjugate to a subgroup of either $\langle A, C\rangle$ or $\langle B, C\rangle$. 
\end{proof}

\begin{lemma}\label{L28} 
The group $\mathrm{GL}(2,\integers)$ has seven conjugacy classes of elements of finite order,  
the class of $I$, three conjugacy classes of elements of order 2, 
represented by $-I, A, C$, one conjugacy class of elements of order 3 represented by $(BC)^2$, 
one conjugacy class of elements of order 4 represented by $AC$, 
and one conjugacy class of element of order 6 represented by $BC$. 
\end{lemma}
\begin{proof} By Lemma \ref{L27}, an element of $\mathrm{GL}(2,\integers)$ of finite order 
is conjugate to either $I, -I, A, C, AC, B, BC$ or $(BC)^2$.  
As $-I$ commutes with every element of $\mathrm{GL}(2,\integers)$, 
we have that $-I$ is conjugate to only itself. 
The element $A$ is not conjugate to $C$ by a simple proof by contradiction. 
The matrices $B$ and $C$ are conjugate in $\mathrm{GL}(2,\integers)$, since 
$$\left(\begin{array}{rr} 1 & 0 \\ 1 & 1 \end{array}\right)\left(\begin{array}{rr} 0 & 1 \\ 1 & 0 \end{array}\right) 
=\left(\begin{array}{rr} -1 & 1 \\ 0 & 1 \end{array}\right)\left(\begin{array}{rr} 1 & 0 \\ 1 & 1 \end{array}\right).$$
\vspace{-.3in}
\par\end{proof}

\begin{lemma}\label{L29} 
Let $\Mu = \langle e_1+I, e_2+I\rangle$, let $\Kappa = \{(a+I)_\star: a \in E^2\}$, and  
let $A, B, C$ be defined as in Lemma \ref{L27}. 
The group $\mathrm{Aff}(\Mu)$  is the free product of the subgroup $\langle \Kappa, A_\star, C_\star\rangle$ 
and the subgroup $\langle \Kappa, B_\star, C_\star\rangle$ amalgamated 
along the subgroup $\langle \Kappa,  (-I)_\star, C_\star\rangle$.  
Every finite subgroup of $\mathrm{Aff}(\Mu)$ is conjugate to a subgroup of either
$\langle \Kappa, A_\star, C_\star\rangle$ or $\langle \Kappa, B_\star, C_\star\rangle$. 
\end{lemma}
\begin{proof} The map $\eta: \mathrm{Aff}(\Mu) \to \mathrm{GL}(2,\integers)$, defined by 
$\eta((a+A)_\star) = A$ for each $A \in \mathrm{GL}(2,\integers)$, 
is an epimorphism with kernel $\Kappa$ by Lemma \ref{L26}. 
Hence $\mathrm{Aff}(\Mu)$ acts on the Bass-Serre tree of the amalgamated product decomposition of $\mathrm{GL}(2,\integers)$, given in Lemma \ref{L27}, via $\eta$. 
Therefore $\mathrm{Aff}(\Mu)$ is the amalgamated product of the groups
$\eta^{-1}(\langle A, C\rangle) = \langle\Kappa, A_\star, C_\star\rangle$ 
and $\eta^{-1}(\langle B, C\rangle) = \langle \Kappa, B_\star, C_\star \rangle$
along the subgroup $\eta^{-1}(\langle -I, C\rangle) = \langle \Kappa,  (-I)_\star, C_\star \rangle$.  

By the torsion theorem for amalgamated products of groups, every finite subgroup of $\mathrm{Aff}(\Mu)$ 
is conjugate to a subgroup of either $\langle \Kappa, A_\star, C_\star\rangle$ or $\langle \Kappa, B_\star, C_\star \rangle$. 
\end{proof}

Let $\Lambda = \langle e_1+I, e_1/2 + \sqrt{3}e_2/2+I\rangle$. 
Then $\Lambda$ is a 2-space group, and the flat orbifold $E^2/\Lambda$ is a {\it hexagonal torus} $\varhexagon$. 

\begin{lemma}\label{L30} 
Let $\Mu = \langle e_1+I, e_2+I\rangle$, and let $\Lambda = \langle e_1+I, e_1/2 + \sqrt{3}e_2/2+I\rangle$. 
Let $\Kappa = \{(a+I)_\star: a \in E^2\}$, and  
let 
$$A = \left(\begin{array}{rr} -1 & 0 \\ 0 & 1 \end{array}\right),\ \  B = \left(\begin{array}{rr} -1 & 1 \\ 0 & 1 \end{array}\right), 
\ \ C = \left(\begin{array}{rr} 0 & 1 \\ 1 & 0 \end{array}\right), \ \ 
D  = \left(\begin{array}{rr} 1 & -1/2 \\ 0 & \sqrt{3}/2 \end{array}\right). $$
Then $\mathrm{Sym}(\Mu) = \langle \Kappa, A_\star, C_\star\rangle$,  and $D\Mu D^{-1}=\Lambda$, 
and $D_\sharp: \mathrm{Aff}(\Mu) \to \mathrm{Aff}(\Lambda)$ is an isomorphism, 
and $\mathrm{Sym}(\Lambda) = D_\sharp(\langle \Kappa, B_\star, C_\star\rangle)$. 
\end{lemma}
\begin{proof}
Now $\mathrm{Out}_E(\Mu) = \mathrm{Aut}_E(\Mu)$ is a finite subgroup of 
$\mathrm{Aut}(\Mu) = \mathrm{GL}(2,\integers)$ that contains $\langle A, C\rangle$. 
Hence $\mathrm{Aut}_E(\Mu)= \langle A, C\rangle$, since $\langle A, C\rangle$ 
is a maximal finite subgroup of $\mathrm{GL}(2,\integers)$ by Lemma \ref{L27}. 
The group $\Kappa$ is the connected component of $\mathrm{Sym}(\Mu)$ that contains $I_\star$. 
Hence $\mathrm{Sym}(\Mu) = \langle \Kappa, A_\star, C_\star\rangle$ by Theorem 2 of \cite{R-T-I}. 

Clearly $D\Mu D^{-1}=\Lambda$, and so 
$D_\#: \mathrm{Aut}(\Mu) \to \mathrm{Aut}(\Lambda)$ and 
$D_\sharp: \mathrm{Aff}(\Mu) \to \mathrm{Aff}(\Lambda)$ are isomorphisms. 
Now $D_\#(\langle B, C\rangle)$ is a subgroup of $\mathrm{Aut}_E(\Lambda)$,  
since $DBD^{-1}$, $DCD^{-1} \in \mathrm{O}(2)$.  
Hence $D_\#^{-1}(\mathrm{Aut}_E(\Lambda))$ is a finite subgroup of $\mathrm{Aff}(\Mu)$ 
that contains $\langle B, C\rangle$. 
Therefore $D_\#^{-1}(\mathrm{Aut}_E(\Lambda)) = \langle B, C\rangle$,
since $\langle B, C\rangle$ is a maximal finite subgroup of $\mathrm{GL}(2,\integers)$ by Lemma \ref{L27}. 
Hence $D_\sharp^{-1}(\mathrm{Sym}(\Lambda)) = \langle \Kappa, B_\star, C_\star\rangle$ 
by Theorem 2 of \cite{R-T-I} and Lemma 10 of \cite{R-T-I}. 
Therefore $\mathrm{Sym}(\Lambda) = D_\sharp(\langle \Kappa, B_\star, C_\star\rangle)$. 
\end{proof}

The square torus $\Box = E^2/\Mu$ is formed by identifying the opposite sides 
of the square fundamental domain for $\Mu$, with vertices $(\pm 1/2, \pm 1/2)$, by translations. 
By Lemma \ref{L30}, the group $\mathrm{Isom}(\Box) = \mathrm{Sym}(\Mu)$ is the semidirect product of the translation subgroup 
$\Kappa = \{(a+I)_\star: a \in E^2\}$ 
and the dihedral subgroup $\langle A_\star, C_\star\rangle$ of order 8 
induced by the symmetry group $\langle A, C\rangle$ of the square fundamental domain of $\Mu$. 
The elements of $\Kappa$ of order at most 2 form the dihedral group of order 4, 
$$\Kappa_2 = \{I_\star, (e_1/2+I)_\star, (e_2/2+I)_\star, (e_1/2+e_2/2+I)_\star\}. $$

By Lemma \ref{L31} below, 
the group $\mathrm{Isom}(\Box)$ has six conjugacy classes of elements of order 2 represented by 

\begin{enumerate}
\item the {\it horizontal halfturn}  h-rot.\ = $(e_1/2+I)_\star$
(or {\it vertical halfturn} v-rot. = $(e_2/2+I)_\star$), 
\item the {\it antipodal map} 2-sym. = $(e_1/2+e_2/2+I)_\star$, 
\item the {\it  halfturn} 2-rot.\ = $(-I)_\star$, 
\item the {\it  horizontal reflection} h-ref.\ = $A_\star$
(or  {\it vertical reflection} v-ref.\ = $(-A)_\star$), 
\item the {\it horizontal glide-reflection}  h-grf.\ = $(e_1/2 - A)_\star$, 
(or  {\it vertical glide-reflection} v-grf.\  = $(e_2/2+A)_\star$), and 
\item the {\it diagonal reflection} d-ref.\ = $C_\star$
(or perpendicular diagonal reflection e-ref.\ = $(-C)_\star$). 
\end{enumerate}

\begin{lemma}\label{L31} 
Let $k \in E^2$, and let $K \in \mathrm{GL}(2,\integers)$. 
Then there exists $a \in E^2$ such that $(a+I)(k + K)(a+I)^{-1} = K$ 
if and only if $k \in \mathrm{Im}(K-I)$. 
\end{lemma}
\begin{proof}
This follows from the fact that
$$(a+I)(k+K)(-a+I) = a +k-Ka+ K = (I-K)a + k +K.$$
\par
\vspace{-.23in}
\end{proof}

In the following discussion and Table \ref{T17},  
an apostrophe ``$\,\hbox{'}\,$" on a symmetry means that the symmetry is multiplied on the left by h-rot., 
a prime symbol ``$\,'\,$" on a symmetry means that the symmetry is multiplied on the left by v-rot., and  
a double prime symbol ``$\,''\,"$ means that the symmetry is multiplied on the left by 2-sym. 
These alterations do not change the conjugacy class of the symmetry by Lemma \ref{L31}. 
Note that conjugating by d-ref.\ = $C_\star$ transposes the ``h-" and ``v-" prefixes, and so also 
the ``$\,\hbox{'}\,$" and ``$\,'\,$" superscripts. 

By Lemma \ref{L33} below, 
the group $\mathrm{Isom}(\Box)$ has 12 conjugacy classes of dihedral subgroups of order 4, 
the classes of the groups $\Kappa_2 = \{$idt., h-rot., v-rot., 2-sym.$\}$, 
$\{$idt, h-rot., 2-rot., 2-rot.'$\}$, 
$\{$idt, 2-sym., 2-rot., 2-rot.$''$$\}$, 
$\{$idt., h-rot., v-ref., h-grf.$\}$, 
$\{$idt, v-rot, v-ref, v-ref.$'$$\}$, 
$\{$idt, 2-sym., d-ref, d-ref.$''$$\}$, 
$\{$idt, v-rot, h-grf., h-grf.$'$$\}$, 
$\{$idt, 2-sym, v-ref, h-grf.$'$$\}$, 
$\{$idt., 2-rot. h-ref., v-ref.$\}$, 
$\{$idt., 2-rot., h-ref.', h-grf.$\}$, 
$\{$idt., 2-rot., h-grf.$'$, v-grf.'$\}$,  and 
$\{$idt., 2-rot., d-ref., e-ref.$\}$. 

The isometries 

\begin{itemize}
\item 4-rot.\ = $(AC)_\star$, and 

\item 4-sym.\ = $(e_2/2 + C)_\star$ 
\end{itemize}
are symmetries of $\Box$ of order 4, 
and the groups $\langle$h-ref., d-ref.$\rangle$, $\langle$v-grf., d-ref.$\rangle$, 
and $\langle$v-rot., d-ref.$\rangle$ are dihedral subgroups of $\mathrm{Isom}(\Box)$ of order 8. 
Moreover 4-rot.\ = (h-ref.)(d-ref.) and 4-rot.$'$ =(v-grf.)(d-ref.) and 4-sym.\ = (v-rot.)(d-ref.). 


The hexagonal torus $\varhexagon = E^2/\Lambda$ is formed by identifying the opposite sides of the regular hexagon 
fundamental domain, with vertices $(\pm 1/2,\pm \sqrt{3}/6), (0, \pm \sqrt{3}/3)$, by translations. 
By Lemma \ref{L30}, the group $\mathrm{Isom}(\varhexagon)$ is the semidirect product 
of the subgroup $\Kappa = \{(a+I)_\star: a \in E^2\}$ 
and the dihedral subgroup $\langle (DBD^{-1})_\star, (DCD^{-1})_\star\rangle$ of order 12 
induced by the symmetry group $\langle DBD^{-1}, DCD^{-1}\rangle$ of the regular hexagon fundamental domain of $\Lambda$. 

We denote the symmetry of $\varhexagon$ represented by the $180^\circ, 120^\circ, 60^\circ$ rotation 
about the center of the regular hexagon fundamental domain of $\Lambda$ 
by 

\begin{itemize} 
\item 2-rot.\ = $(-I)_\star$, 
\item 3-rot.\ = $(D(BC)^2D^{-1})_\star$ and 
\item 6-rot.\ = $(DBCD^{-1})_\star$ respectively. 
\end{itemize}

Define affinities of $\Box$  
by 

\begin{itemize} 
\item 3-aff.\ $= ((BC)^2)_\star$ and 
\item 6-aff.\ $= (BC)_\star$. 
\end{itemize}
The affinities 3-aff.\ and 6-aff.\ of $\Box$ are conjugate by $D_\star$ 
to the symmetries 3-rot.\ and 6-rot.\ of $\varhexagon$ respectively. 

Define the reflection l-ref.\ of $\varhexagon$ by  
\begin{itemize}
\item l-ref.\ = $(DCD^{-1})_\star$. 
\end{itemize}
Then l.ref.\ is represented by the reflection in the line segment joining 
the midpoints $\pm (1/4, \sqrt{3}/4)$ of two opposite sides of $\varhexagon$. 

Define the reflection m-ref.\ of $\varhexagon$ by 
\begin{itemize} 
\item m-ref.\ = $(DBD^{-1})_\star = A_\star$.  
\end{itemize}
Then m-ref.\ is represented by the reflection in the line segment joining 
the opposite vertices $(0, \pm \sqrt{3}/3))$ of $\varhexagon$. 

Define the reflections n-ref.\  and o-ref.\ of $\varhexagon$ by
\begin{itemize}
\item  n-ref.\ = $(D(BC)^2CD^{-1})_\star = (DBCBD^{-1})_\star$ and 
\item o-ref.\ = $(D(BC)^3CD^{-1})_\star = (D(-C)D^{-1})_\star$.
\end{itemize} 
The groups $\langle$l-ref., n-ref.$\rangle$ and $\langle$m-ref., o-ref.$\rangle$ are dihedral groups of order 6   
with 3-rot.\ = (n-ref.)(l-ref.) = (o-ref.)(m-ref.). 
The group $\langle$l-ref., m-ref.$\rangle$ is a dihedral group of order 12 with 6-rot.\ = (m-ref.)(l-ref.). 
The symmetries 2-rot., d-ref., e-ref.\ of $\Box$ are conjugate by $D_\star$ to the symmetries 2-rot., l-ref., o-ref.\  of $\varhexagon$ respectively. 

Define affinities of $\Box$ 
by 
\begin{itemize}
\item m-aff.\ = $B_\star$ and 
\item n-aff.\ = $(BCB)_\star$. 
\end{itemize}
The affinities m-aff.\ and n-aff.\ of $\Box$ are conjugate by $D_\star$ 
to the symmetries m-ref.\ and n-ref.\ of $\varhexagon$, respectively. 
The groups $\langle$d-ref., n-aff.$\rangle$, $\langle$m-aff., e-ref.$\rangle$, $\langle$d-ref., m-aff.$\rangle$ 
are conjugate by $D_\star$ to the groups $\langle$l-ref., n-ref.$\rangle$, $\langle$m-ref., o-ref.$\rangle$, 
$\langle$l-ref., m-ref.$\rangle$ respectively.

\begin{lemma}\label{L32} 
Let $\Mu = \langle e_1+I, e_2+I\rangle$, and let $\kappa = (k+K)_\star$ be an element of $\mathrm{Aff}(\Mu)$ 
such that $K$ has finite order.  Then $K$ is conjugate to exactly one of  $I, -I, A, C, (BC)^2, AC$ or $BC$. 
\begin{enumerate}
\item If $K = I$ and $\kappa$ has order 2, then $\kappa \in \Kappa_2$ and $\kappa$ is conjugate to {\rm h-rot.\ =} $(e_1/2+I)_\star$. 
\item If $K =  -I$, then $\kappa$ is conjugate to {\rm 2-rot.\ =} $(-I)_\star$. 
\item If $K$ is conjugate to $A$ and $\kappa$ has order 2, 
then $\kappa$ is conjugate to exactly one of {\rm  h-ref.\ =} $A_\star$ or 
{\rm  v-grf.\ =} $(e_2/2 +A)_\star$.
\item If $K$ is conjugate to $C$ and $\kappa$ has order 2 , then $\kappa$ is conjugate to {\rm  d-ref.\ =} $C_\star$.
\item If $K$ is conjugate to $(BC)^2$, then $\kappa$ is conjugate to {\rm  3-aff.\ =} $(BC)^2_\star$. 
\item If $K$ is conjugate to $AC$, then $\kappa$ is conjugate to {\rm  4-rot.\ =} $(AC)_\star$. 
\item If $K$ is conjugate to $BC$, then $\kappa$ is conjugate to {\rm  6-aff.\ =} $(BC)_\star$. 
\end{enumerate}
\end{lemma}
\begin{proof} The matrix $K$ is conjugate to exactly one of  $I, -I, A, C, (BC)^2, AC$ or $BC$ by Lemma \ref{L28}. 
Hence, we may assume that $K = I, -I, A, C, (BC)^2, AC$ or $BC$. 

(1)  As $\kappa^2 = 2k+I$, we have that $2k \in \integers^2$.  Hence $\kappa \in \Kappa_2$. 
We have that $C(e_1/2+I)C^{-1} = e_2/2+I$ and $B(e_2/2+I)B^{-1} = e_1/2+e_2/2$. 
Hence $\kappa$ is conjugate to $(e_1/2+I)_\star$. 

(2) We have that $(k/2+I)(-I)(-k/2+I) = k-I$. 

(3) Observe that $(k+A)^2 = 2k_2e_2+I$, and so $2k_2\in \integers$.  
Hence $\kappa = (k_1e_1+A)_\star$ or  $\kappa = (e_2/2+k_1e_1+A)_\star$. 
We have that $(k_1e_1/2+I)A(-k_1e_1/2+I) = k_1e_1+A$, and so
$(k_1e_1/2+I)(e_2/2+A)(-k_1e_1/2+I) = e_2/2+k_1e_1+A$. 
The isometries h-ref.\ and v-grf.\,are not conjugate in $\mathrm{Aff}(\Mu)$, 
since h-ref.\ fixes points of $\Box$ whereas v-grf.\ does not. 

(4) We have that $(k + C)^2 = (k_1+k_2)e_1+(k_1+k_2)e_2+I$. 
Hence $k_1+k_2 \in \integers$, and so $\kappa = (k_1e_1-k_1e_2 +C)_\star$. 
We have that $\kappa$ is conjugate to $C_\star$ by Lemma \ref{L31}. 

(5)  We have that $\kappa$ is conjugate to $((BC)^2)_\star$ by Lemma \ref{L31}. 
(6) The proofs of cases (6) and (7) are similar to the proof of (5). 
\end{proof}

\begin{lemma}\label{L33}  
Let $\Mu = \langle e_1+I,e_2+I\rangle$. 
Let $\kappa = (k+K)_\star$ and $\lambda = (\ell +L)_\star$ be distinct elements of order 2 of $\mathrm{Aff}(\Mu)$ 
such that $\Omega(\kappa\lambda)$ has finite order in $\mathrm{Out}(\Mu)$. 
Let $\Eta = \langle K, L \rangle$.  
Then $K$ and $L$ have order 1 or 2 and either $\Eta$ is a cyclic group of order 1 or 2 or $\Eta$ is a dihedral group 
of order 4, 6, 8 or 12. 
\begin{enumerate}
\item If $\Eta = \{I\}$, then $\langle \kappa, \lambda\rangle = \Kappa_2$ 
and $\{\kappa, \lambda\}$ is conjugate to $\{${\rm h-rot.,  v-rot.}$\}$.
\item If $\Eta$ has order 2, then $\Eta$ is conjugate to exactly one of $\langle -I\rangle$, $\langle CA\rangle$
 or $\langle C\rangle$. 
\begin{enumerate}
\item If $K = I$ and $L = -I$, then $\{\kappa, \lambda\}$ is conjugate to $\{${\rm h-rot.,  2-rot.}$\}$. 
\item If $K = L =-I$, then $\{\kappa, \lambda\}$ is conjugate to 
$\{${\rm 2-rot.}, $(v+I)_\star${\rm 2-rot.}$\}$ with $Kv = -v$.  
\item If $K = I$ and $L$ is conjugate to $A$, then $\{\kappa, \lambda\}$ is conjugate to exactly one of 
$\{${\rm h-rot.,  v-ref.}$\}$, $\{${\rm v-rot.,  v-ref.}$\}$, $\{${\rm 2-sym.,  v-ref.}$\}$, $\{${\rm h-rot.,  h-grf.}$\}$, 
$\{${\rm v-rot.,  h-grf.}$\}$ or $\{${\rm 2-sym.,  h-grf.}$\}$. 
\item If $K = L$ and $L$ is conjugate to $A$, then $\{\kappa, \lambda\}$ is conjugate to 
exactly one of $\{${\rm h-ref.}, $(v+I)_\star${\rm h-ref.}$\}$, $\{${\rm h-ref.}, $(v+I)_\star${\rm v-grf.}$\}$, 
or $\{${\rm v-grf.}, $(v+I)_\star${\rm v-grf.}$\}$ with $Av = -v$.  
\item If $K = I$ and $L$ is conjugate to $C$, then $\{\kappa, \lambda\}$ is conjugate to exactly one of 
$\{${\rm 2-sym.,  d-ref.}$\}$ or $\{${\rm v-rot.,  d-ref.}$\}$. 
\item If $K = L$ and $L$ is conjugate to $C$, then $\{\kappa, \lambda\}$ is conjugate to 
$\{${\rm d-ref.}, $(v+I)_\star${\rm d-ref.}$\}$ with $Cv = -v$. 
\end{enumerate}
\item If $\Eta$ has order 4,  then  $\Eta$ is conjugate to exactly one of  $\langle -I, A \rangle$ or $\langle -I, C\rangle$. 
\begin{enumerate}
\item If $K = -I$ and $L$ is conjugate to $A$, then $\{\kappa, \lambda\}$ is conjugate to exactly one of  
$\{${\rm 2-rot.}, $(v+I)_\star${\rm h-ref.}$\}$ or $\{${\rm 2-rot.}, $(v+I)_\star${\rm v-grf.}$\}$ with $Av = -v$. 
\item If $K = -L$ and $K$ is conjugate to $A$, then $\{\kappa, \lambda\}$ is conjugate to exactly one of 
$\{${\rm h-ref., v-ref.}$\}$,  $\{${\rm v-ref.}$'$, {\rm v-grf.}$\}$ or $\{${\rm v-grf.'}, {\rm h-grf.}$'$$\}$. 
\item If $K = -I$ and $L$ is conjugate to $C$, then $\{\kappa, \lambda\}$ is conjugate to 
$\{${\rm 2-rot.}, $(v+I)_\star${\rm d-ref.}$\}$ with $Cv = -v$. 
\item If $K = -L$ and $K$ is conjugate to $C$, then $\{\kappa, \lambda\}$ is conjugate to $\{${\rm d-ref., e-ref.}$\}$. 
\end{enumerate}
\item If $\Eta$ has order 8,  then $\Eta$ is conjugate to $\langle A, C\rangle$ and 
$\{\kappa, \lambda\}$ is conjugate to exactly one of $\{${\rm h-ref., d-ref.}$\}$ or $\{${\rm v-grf., d-ref.}$\}$. 
\item If $\Eta$ has order 6,  then $\Eta$ is conjugate to exactly one of $\langle BCB, C\rangle$ or $\langle -C, B\rangle$. 
\begin{enumerate}
\item In the former case, $\{\kappa, \lambda\}$ is conjugate to $\{${\rm n-aff., d-ref.}$\}$.
\item In the latter case, $\{\kappa, \lambda\}$ is conjugate to $\{${\rm e-ref., m-aff.}$\}$. 
\end{enumerate}
\item If $\Eta$ has order 12, then $\Eta$ is conjugate to $\langle B, C\rangle$ and 
$\{\kappa, \lambda\}$ is conjugate to $\{${\rm  m-aff., d-ref.}$\}$.
\end{enumerate}
\end{lemma}
\begin{proof}
As $\kappa^2 = (k +Kk + K^2)_\star$, we have that $k + Kk \in \integers^2$ and $K^2 = I$.  
Hence $K$ has order 1 or 2. Likewise $L$ has order 1 or 2. 
Now $KL$ has finite order by Theorem 3 of \cite{R-T-I} and Lemma \ref{L26}, 
moreover $KL$ has order 1, 2, 3, 4 or 6 by Lemma \ref{L32}. 
Hence, either $\Eta$ is cyclic of order 1 or 2 or $\Eta$ is dihedral of order 4, 6, 8 or 12.

(1) We have that $\langle \kappa,\lambda\rangle = \Kappa_2$ by Lemma \ref{L32}(1),  
moreover $B_\star\{$h-rot., v-rot.$\}B^{-1}_\star = \{$h-rot., 2-sym.$\}$ and 
$C_\star\{$h-rot., 2-sym.$\}C_\star^{-1} = \{$v-rot., 2-sym.$\}$. 

(2)  If $\Eta$ has order 2,  then $\Eta$ is conjugate to exactly one of $\langle -I\rangle$, $\langle A\rangle$
 or $\langle C\rangle$ by Lemma \ref{L28}; moreover if $\Eta$ is conjugate to $\langle -I\rangle$, then $\Eta = \langle -I\rangle$. 
 
(a) By conjugating $\{\kappa, \lambda\}$, we may assume that $\lambda$ = 2-rot.\,by Lemma \ref{L32}(2). 
Then $\{\kappa, \lambda\}$ is conjugate to $\{${\rm h-rot.,  2-rot.}$\}$ as in the proof of Lemma \ref{L32}(1). 

(b) By Lemma \ref{L32}(2), we may assume that $\kappa = K_\star = $ 2-rot. 
Then $L = (\ell + K)_\star = (\ell+I)_\star$2-rot.\ with $K\ell = -\ell$. 

(c) By conjugating $\{\kappa, \lambda\}$, we may assume that $\lambda$ = h-ref.\ or v-grf.\ by Lemma \ref{L32}(3). 
Then $\kappa \in \Kappa_2$ by Lemma \ref{L32}(1), and so there are six possibilities.  
These six pairs are nonconjugate since h-ref.\ is not conjugate to v-grf., the centralizer of $A$ in $\mathrm{GL}(2,\integers)$ 
is $\langle -I, A\rangle$, and $\langle -I, A\rangle$ centralizes $\Kappa_2$. 

(d) By Lemma \ref{L32}(3),  we may assume that $K = A$ and $\kappa$ = h-ref.\ or v-grf. 
Now $\lambda = (v +A)_\star$ or $\lambda = (e_2/2 + v + A)_\star$ with $Av = -v$ 
by the proof of Lemma \ref{L32}(3). 
Hence $\lambda = (v+I)_\star$h-ref. or $\lambda = (v+I)_\star$v-grf.\, with $Av = - v$. 
Moreover $(-v/2+I)_\star\{$h-ref., $(v+I)_\star$v-grf.$\}(v/2+I)_\star = \{(-v+I)_\star$h-ref., v-grf.$\}$. 

(e) By conjugating $\{\kappa, \lambda\}$, we may assume that $\lambda$ = d-ref.\ by Lemma \ref{L32}(4). 
Then $\kappa \in \Kappa_2$ by Lemma \ref{L32}(1), and so there are three possibilities.  
We have that $C(e_1/2+C)C^{-1} = e_2/2 + C$.  
The pair $\{${\rm v-rot.,  d-ref.}$\}$  is not conjugate to $\{${\rm 2-sym,  d-ref.}$\}$, 
since (v-rot.)(d-ref.) has order 4 while (2-sym.)(d-ref.) has order 2. 

(f) By Lemma \ref{L32}(4),  we may assume that $K = C$ and $\kappa = C_\star =$ d-ref. 
Then $\lambda = (\ell + C)_\star = (\ell + I)_\star$d-ref.\ with $C\ell = -\ell$ by the proof of Lemma \ref{L32}(4).  

(3)  If $\Eta$ has order 4, then $\Eta$ is conjugate to exactly one of $\langle -I, A\rangle$ or $\langle -I, C\rangle$ 
by Lemma \ref{L27}.

(a) By conjugating $\{\kappa, \lambda\}$, we may assume that $L = A$ and $\kappa$ = 2-rot.\,by Lemma \ref{L31}. 
Now $\lambda = (v+A)_\star$ or $\lambda = (e_2/2+v+A)_\star$ with $Av = -v$ by the proof of Lemma \ref{L32}(3).   
Hence $\lambda = (v+I)_\star$h-ref.\ or $\lambda = (v+I)_\star$v-grf.\ with $Av= -v$.

(b) By conjugating $\{\kappa, \lambda\}$, we may assume that $K = A$ and $\kappa\lambda$ = 2-rot. 
By the proof of Lemma \ref{L32}(3), we have that $\kappa = (k_1e_1 +A)_\star$ or $(e_2/2+k_1e_1 + A)_\star$ 
and $\lambda = (\ell_2e_2-A)_\star$ or $(e_1/2+\ell_2e_2 -A)_\star$.  As $\kappa\lambda$ = 2-rot., 
we have that $\{\kappa, \lambda\}$ is either $\{$h-ref., v-ref.$\}$,  $\{$h-ref.', h-grf.$\}$, 
$\{$v-grf., v-ref.$'\}$, or $\{$v-grf.', h-grf.$'\}$.  Moreover $C_\star\{$h-ref.', h-grf.$\}C_\star^{-1} = \{$v-ref.$'$, v-grf.$\}$. 

(c) By conjugating $\{\kappa, \lambda\}$, we may assume that $L = C$ and $\kappa$ = 2-rot. 
By the proof of Lemma \ref{L32}(4), we have that $\lambda = (v+C)_\star = (v+I)_\star$d-ref.\ 
with $Cv = -v$. 

(d) By conjugating $\{\kappa, \lambda\}$, we may assume that $K = C$ and $\kappa\lambda$ = 2-rot. 
By the proof of Lemma \ref{L32}(4), we have that $\kappa = (k_1e_1 - k_1e_2 +C)_\star$.  
Now $(\ell - C)^2 = (\ell_1-\ell_2)e_1 + (\ell_2-\ell_1)e_2 + I$.  
Hence $\ell_1-\ell_2 \in \integers$, and so $\lambda = (\ell_1e_1 + \ell_1e_2-C)_\star$. 
Now $(k_1e_1 - k_1e_2 +C)(\ell_1e_1 + \ell_1e_2-C) = (k_1+\ell_1)e_1 + (\ell_1-k_1)e_1 - I$. 
Hence $k_1+\ell_1, \ell_1-k_1 \in \integers$.  Thus we may take $\ell_1 = k_1$ and $k_1 = 0$ or $1/2$. 
Then $\{\kappa, \lambda\}$ is either $\{$d-ref., e-ref.$\}$ or $\{$d-ref.$''$, e-ref.$''\}$.  
Moreover $(e_1/2+I)C(-e_1/2+I) = e_1/2-e_2/2+C$ and $(e_1/2+I)(-C)(-e_1/2+I) = e_1/2+e_2/2-C$.  

(4)  If $\Eta$ has order 8,  then $\Eta$ is conjugate to $\langle A, C\rangle$ by Lemma \ref{L27}. 
By conjugating $\{\kappa, \lambda\}$, we may assume that $K = A$ and $L = C$ and $\kappa\lambda$ = 4-rot. 
By Lemma \ref{L32}(3), we have that $\kappa = (k_1e_1+A)_\star$ or $(k_1e_1+e_2/2+A)_\star$, 
and by Lemma \ref{L32}(4), we have that $\lambda = (\ell_1e_1-\ell_1e_2+C)_\star$. 
Now $(k_1e_1 +A)(\ell_1e_1 - \ell_1e_2+C) = (k_1-\ell_1)e_1 - \ell_1e_2 + AC$, 
and we may take $k_1 = \ell_1 = 0$, 
moreover $(k_1e_1 +e_2/2 +A)(\ell_1e_1-\ell_1e_2+C) = (k_1-\ell_1)e_1 +(1/2- \ell_1)e_2 + AC$, 
and we may take $k_1 = \ell_1 = 1/2$. 
Then  we have that $\{\kappa, \lambda\}$ = $\{$h-ref., d-ref.$\}$ or $\{$v-grf.', d-ref.$''$$\}$. 
Moreover, we have that $(e_1/4-e_2/4+I)(e_2/2+A)(-e_1/4+e_2/4+I) = e_1/2+e_2/2+A$ and 
we have that $(e_1/4-e_2/4+I)C(-e_1/4+e_2/4+I) = e_1/2 -e_2/2+C$. 

(5) If $\Eta$ has order 6,  then $\Eta$ is conjugate to exactly one of 
$\langle B^2, C\rangle$ or $\langle B^2, CB\rangle$ by Lemma \ref{L27}. 

(a) By conjugating $\{\kappa, \lambda\}$, we may assume that $K = BCB$ and $L = C$ and $\kappa\lambda$ = 3-aff. 
Now $(k + BCB)^2 = (2k_2-k_1)e_2 + I$. Hence $\kappa = 2k_2e_1+k_2e_2 + BCB$. 
By Lemma \ref{L32}(4), we have that $\lambda = (\ell_1e_1-\ell_1e_2+C)_\star$. 
Now, we have that $\kappa\lambda = ((2k_2-\ell_1)e_1+(k_2 - 2\ell_1)e_2 +(BC)^2)_\star$.  
Hence, we may take $k_2 = \ell_1 = 0$, and so $\kappa = (BCB)_\star$ = n-aff.\ and $\lambda = C_\star$ = d-ref.

(b) By conjugating $\{\kappa, \lambda\}$, we may assume that $K = -C$ and $L = B$ and $\kappa\lambda$ = 3-aff. 
Now $(k-C)^2 = (k_1 - k_2)e_1+(k_2 - k_1)e_2+ I$. 
Hence $\kappa = (k_1e_1+k_1e_2 - C)_\star$. 
Now $(\ell + B)^2 = \ell_2e_1+2\ell_2e_2+I$. 
Hence $\lambda= B_\star$ = m-aff. 
Now $(k_1e_1+k_1e_2 - C)B = k_1e_1+k_1e_2 + (BC)^2$. 
Hence $\kappa = (-C)_\star$ = e-ref.

(6) If $\Eta$ has order 12,  then $\Eta$ is conjugate to $\langle B, C\rangle$ by Lemma \ref{L27}. 
By conjugating $\{\kappa, \lambda\}$, we may assume that $K = B$ and $L = C$ and $\kappa\lambda$ = 6-aff. 
By the proof of 5(b), we have that $\kappa = B_\star$ = m-aff. 
By Lemma \ref{L32}(4), we have that $\lambda = (k_1e_1-k_1e_2+C)_\star$. 
Now $B(k_1e_1-k_1e_2+C) = -2k_1e_1 - k_1e_2 + BC$. 
Hence $\lambda = C_\star$ = d-ref.
\end{proof}

\begin{table} 
\begin{tabular}{rlllll}
no. & fibers & grp. & quotients &  structure group action & classifying pair \\
\hline 
    1 & $(\circ, \mathrm{O})$ & $C_1$ & $(\circ, \mathrm{O})$ & (idt., idt.)  & \{idt., idt.\}  \\
    2 & $(\circ, \mathrm{O})$ & $C_2$ & $(2222, \mathrm{I})$ & (2-rot., ref.) & \{2-rot., 2-rot.$\}_\ast$ \\
    3 & $(\circ, \mathrm{O})$ & $C_2$ & $(\ast\ast, \mathrm{I})$ & (v-ref., ref.) & \{v-ref., v-ref.$\}_\ast$ \\
    4 & $(\circ, \mathrm{O})$& $C_2$ & $(\times\times, \mathrm{I})$ & (h-grf., ref.) & \{h-grf., h-grf.$\}_\ast$\\
    4 & $(\circ, \mathrm{O})$ & $ C_2$ & $(2222, \mathrm{O})$ & (2-rot., 2-rot.) & \{2-rot., 2-rot.\} \\
    5 & $(\circ, \mathrm{O})$ & $C_2$ & $(\ast\times, \mathrm{I})$ & (e-ref., ref.) & \{e-ref., e-ref.$\}_\ast$\\
    5 & $(\circ, \mathrm{O})$ & $D_2$ & $(\ast\ast, \mathrm{I})$ & (v-ref., ref.), (h-rot., 2-rot.) & \{v-ref., h-grf.$\}_\ast$ \\
    6 & $(\circ, \mathrm{I})$  & $C_1$ & $(\circ, \mathrm{I})$ & (idt., idt.)  & \{idt., idt.\}  \\
    7 & $(\circ, \mathrm{O})$ & $C_2$ & $(\circ, \mathrm{I})$ & (v-rot., ref.) & \{v-rot., v-rot.\} \\
    7 & $(\circ, \mathrm{O})$ & $C_2$ & $(\ast\ast, \mathrm{O})$ & (v-ref., 2-rot.) & \{v-ref., v-ref.\}\\
    8 & $(\circ, \mathrm{I})$ & $D_1$ & $(\circ, \mathrm{I})$ & (h-rot., ref.) & \{idt., h-rot\} \\
    9 & $(\circ, \mathrm{O})$ & $D_2$ & $(\circ, \mathrm{I})$ & (v-rot., ref.), (h-rot., 2-rot.) &  \{v-rot, 2-sym.\} \\
    9 & $(\circ, \mathrm{O})$ & $C_2$ & $(\ast\times, \mathrm{O})$ & (d-ref., 2-rot.)  & \{d-ref., d-ref.\} \\
  11 & $(\circ, \mathrm{I})$ & $D_1$ & $(2222, \mathrm{I})$ &(2-rot., ref.)  & \{idt., 2-rot.\} \\
  13 & $(\circ, \mathrm{O})$ & $D_2$ & $(\ast 2222, \mathrm{I})$ & (h-ref., ref.),  (v-ref., 2-rot.) & \{h-ref., 2-rot.$\}_\ast$ \\
  14 & $(\circ, \mathrm{O})$ & $D_2$ & $(2222, \mathrm{I})$ & (v-rot., ref.),  (2-rot.$'$, 2-rot.) & \{v-rot., 2-rot.\} \\
  14 & $(\circ, \mathrm{O})$ & $D_2$ & $(22\ast, \mathrm{I})$ & (v-grf., ref.), (v-ref.$'$, 2-rot.) & \{v-grf., 2-rot.$\}_\ast$ \\
  15 & $(\circ, \mathrm{O})$ & $D_2$ & $(2{\ast}22, \mathrm{I})$ & (e-ref., ref.), (d-ref., 2-rot.) & \{e-ref., 2-rot.$\}_\ast$ \\
  17 & $(\circ, \mathrm{O})$ & $D_2$ & $(\ast 2222, \mathrm{I})$ & (v-ref., ref.), (2-rot., 2-rot.) &  \{v-ref., h-ref.\} \\
  18 & $(\circ, \mathrm{O})$ & $D_2$ & $(22\ast, \mathrm{I})$ & (h-grf., ref.), (2-rot.', 2-rot.)  & \{h-grf., h-ref.\} \\
  19 & $(\circ, \mathrm{O})$ & $D_2$ & $(22\times, \mathrm{I})$ & (v-grf., ref.), (2-rot.$''$, 2-rot.) &  \{v-grf., h-grf.\} \\
  20 & $(\circ, \mathrm{O})$ & $D_2$ & $(2{\ast}22, \mathrm{I})$ & (d-ref., ref.), (2-rot., 2-rot.)  & \{d-ref., e-ref.\} \\
  28 & $(\circ, \mathrm{I})$ & $D_1$ & $({\ast}{\ast}, \mathrm{I})$ & (h-ref., ref.) &  \{idt., h-ref.\} \\
  29 & $(\circ, \mathrm{O})$ & $D_2$ & $({\ast}{\ast}, \mathrm{I})$ & (v-rot., ref.), (h-ref., 2-rot.) & \{v-rot., v-grf.\} \\
  30 & $(\circ, \mathrm{O})$ & $D_2$ & $({\ast}{\ast}, \mathrm{I})$ & (v-rot., ref.), (v-grf., 2-rot.) & \{v-rot., h-ref.\} \\
  31 & $(\circ, \mathrm{I})$ & $D_1$ & $({\times}{\times}, \mathrm{I})$ & (v-grf., ref.)  & \{idt., v-grf.\} \\
  32 & $(\circ, \mathrm{O})$ & $D_2$ & $({\ast}{\ast}, \mathrm{I})$ & (h-rot., ref.), (h-ref.', 2-rot.) & \{h-rot., h-ref.\} \\
  33 & $(\circ, \mathrm{O})$ & $D_2$ & $({\times}{\times}, \mathrm{I})$ & (h-rot., ref.), (v-grf.', 2-rot.) &  \{h-rot., v-grf.\} \\
  33 & $(\circ, \mathrm{O})$ & $D_2$ & $({\ast}{\times}, \mathrm{I})$ & (2-sym., ref.), (h-ref., 2-rot.) & \{2-sym., v-grf.'\} \\
  34 & $(\circ, \mathrm{O})$ & $D_2$ & $({\ast}{\times}, \mathrm{I})$ & (2-sym., ref.), (v-grf., 2-rot.) & \{2-sym., h-ref.'\} \\
  40 & $(\circ, \mathrm{I})$ & $D_1$ & $({\ast}{\times}, \mathrm{I})$ & (e-ref., ref.)  &  \{idt., e-ref.\} \\
  41 & $(\circ, \mathrm{O})$ & $D_2$ & $({\ast}{\ast}, \mathrm{I})$ & (e-ref.$''$, ref.), (e-ref., 2-rot.)  & \{e-ref.$''$, 2-sym.\} \\
  43 & $(\circ, \mathrm{O})$ & $D_4$ & $({\ast}{\times}, \mathrm{I})$ & (v-rot., ref.), (4-sym., 4-rot.) &  \{v-rot., d-ref.\} \\
  76 & $(\circ, \mathrm{O})$ & $C_4$ & $(442, \mathrm{O})$ & (4-rot., 4-rot.)  & \{4-rot., 4-rot.$^{-1}$\} \\
  91 & $(\circ, \mathrm{O})$ & $D_4$ & $({\ast}442, \mathrm{I})$ & (h-ref., ref.), (4-rot., 4-rot.) & \{h-ref., d-ref.\} \\
  92 & $(\circ, \mathrm{O})$ & $D_4$ & $(4{\ast}2, \mathrm{I})$ & (v-grf., ref.), (4-rot.$'$, 4-rot.) & \{v-grf., d-ref.\}  \\
144 & $(\circ, \mathrm{O})$ & $C_3$ & $(333, \mathrm{O})$ & (3-rot., 3-rot.)  &  \{3-aff., 3-aff.$^{-1}$\} \\
151 & $(\circ, \mathrm{O})$ & $D_3$ & $({\ast}333, \mathrm{I})$ & (o-ref., ref.), (3-rot., 3-rot.)  & \{e-ref., m-aff.\} \\
152 & $(\circ, \mathrm{O})$ & $D_3$ & $(3{\ast}3, \mathrm{I})$ & (n-ref., ref.), (3-rot., 3-rot.) & \{n-aff., d-ref.\} \\
169 & $(\circ, \mathrm{O})$ & $C_6$ & $(632, \mathrm{O})$ & (6-rot., 6-rot.)  & \{6-aff., 6-aff.$^{-1}$\} \\
178 & $(\circ, \mathrm{O})$ & $D_6$ & $({\ast}632, \mathrm{I})$ & (m-ref., ref.), (6-rot., 6-rot.) & \{m-aff., d-ref.\} \\
\end{tabular}

\medskip
\caption{The classification of the co-Seifert fibrations of 3-space groups 
whose generic fiber is of type $\circ$ with IT number 1}\label{T17}
\end{table}

We represent $\mathrm{Out}(\Mu)$ by the image of the monomorphism 
$\sigma: \mathrm{GL}(2,\integers) \to \mathrm{Aff}(\Mu)$ (see Lemmas \ref{L25} and \ref{L26}).  
The set $\mathrm{Isom}(C_\infty,\Mu)$ consists of seven elements 
corresponding to the pairs of inverse elements \{idt., idt.\}, \{2-rot., 2-rot.\}, \{h-ref., h-ref.\}, \{d-ref., d-ref.\}, \{4-rot., 4-rot.$^{-1}$\}, \{3-aff., 3-aff.$^{-1}$\}, \{6-aff., 6-aff.$^{-1}$\} of $\sigma(\mathrm{GL}(2, \integers))$ 
by Lemma \ref{L28} and Theorem 7 of \cite{R-T-I}. 
The pairs  \{h-ref., h-ref.\} and  \{v-ref., v-ref.\}  determine the same element of $\mathrm{Isom}(C_\infty,\Mu)$, 
since they are conjugate. 

The set $\mathrm{Isom}(D_\infty,\Mu)$ consists of thirty four elements corresponding to the remaining classifying pairs 
of elements of $\mathrm{Aff}(\Mu)$ in Table \ref{T17} by Lemma \ref{L33} and Theorems 9 and 10 of \cite{R-T-I}. 
Note that one may have to conjugate a classifying pair in Table \ref{T17} to make it correspond to a pair in Lemma \ref{L33}. 

\section{Enantiomorphic 3-space group pairs}

In this section, we apply our theory 
to give an explanation for all but one of the enantiomorphic 3-space group pairs.   
An enantiomorphic 3-space group pair consists of a pair $(\Gamma_1, \Gamma_2)$ of isomorphic 
3-space groups all of whose elements are orientation-preserving such that there is no orientation-preserving 
affinity of $E^3$ that conjugates one to the other. 
There are 11 enantiomorphic 3-space group pairs up to isomorphism, 
and 10 of them have 2-dimensional, complete, normal subgroups; 
moreover, these complete normal subgroups are unique.  

The first enantiomorphic pair $(\Gamma_1, \Gamma_2)$ has IT numbers 76 and 78. 
The geometric co-Seifert fibration of $E^3/\Gamma_1$ is described in Table \ref{T17}. 
The action of the cyclic structure group of order 4 on $\circ \times \mathrm{O}$ is given by $($4-rot., 4-rot.$)$. 
The action of the structure group for $E^3/\Gamma_2$ 
is given by $($4-rot.$^{-1}$, 4-rot.), 
since the classifying pair is $\{$4-rot., 4-rot.$^{-1}\}$.  

Let $\Nu_i$ be the unique 2-dimensional, complete, normal subgroup of $\Gamma_i$ for $i = 1, 2$. 
We have that $\mathrm{Span}(\Nu_i) = E^2$. 
Let $\phi$ be an affinity of $E^3$ such that $\phi\Gamma_1\phi^{-1} = \Gamma_2$, 
Then $\phi \Nu_1\phi^{-1} = \Nu_2$, since $\phi\Nu_1\phi^{-1}$ is a 2-dimensional, complete, normal subgroup of $\Gamma_2$.  
Let $\ov \phi: E^2 \to E^2$ be the restriction of $\phi$, and 
let $\phi' : (E^2)^\perp \to (E^2)^\perp$ be the restriction of $\phi$ followed by orthogonal projection. 
Then $\phi$ is orientation-preserving if and only if either $\ov\phi$ and $\phi'$ are both orientation-preserving 
(orientation-preserving case)
or $\ov\phi$ and $\phi'$ are both orientation-reversing (orientation-reversing case). 

We now prove that the pair $(\Gamma_1, \Gamma_2)$ is enantiomorphic. 
Assume that $\phi$ is orientation-preserving. 
By Theorem 3.3 of \cite{R-T-B}, we have the equation
$$\Xi_2\Rho_2^{-1}(\phi')_\ast\Rho_1 = (Dp_1)_\star(\ov\phi)_\sharp\Xi_1.$$
The map $\Rho_2^{-1}(\phi')_\ast\Rho_1:\Gamma_1/\Nu_1 \to \Gamma_2/\Nu_2$ is an isomorphism 
of infinite cyclic groups that maps the positively oriented generator of $\Gamma_1/\Nu_1$ 
to the positively oriented generator of $ \Gamma_2/\Nu_2$ if and only if $\phi'$ is orientation-preserving. 
The map $\Xi_i: \Gamma_i/\Nu_i \to \mathrm{Isom}(E^2/\Nu_i)$ is the homomorphism induced 
by the action of $\Gamma_i/\Nu_i$ on $E^2/\Nu_i$, which factors through the action of the structure 
group on $E^2/\Nu_i$ for each $i = 1, 2$. 
The image of the positively oriented generator of $\Gamma_1/\Nu_1$ under $\Xi_2\Rho_2^{-1}(\phi')_\ast\Rho_1$ 
is therefore (4-rot.)$^{-1} = (AC)^{-1}_\star$ in the orientation-preserving case and 4-rot. $=(AC)_\star$ in the orientation-reversing case. 

The map $(Dp_1)_\star(\ov\phi)_\sharp\Xi_1$ is harder to evaluate because $(Dp_1)_\star$ is a crossed homomorphism from $\Gamma_1/\Nu_1$ to the connected component $\mathcal{K}_2$ of the identity of the Lie group 
$\mathrm{Isom}(E^2/\Nu_2)$. To simplify, we apply the epimorphism $\Omega: \mathrm{Aff}(E^2/\Nu_2) \to 
\mathrm{Out}(\Nu_2)$ whose kernel is $\mathcal{K}_2$ by Theorem 3 of \cite{R-T-I}. 
This cancels the action of $(Dp_1)_\star$. 
To simply further, we identify $\mathrm{Out}(\Nu_2)$ with $\mathrm{GL}(2,\integers)$ as in Lemma \ref{L25}. 
The map $(\ov\phi)_\sharp: \mathrm{Aff}(E^2/\Nu_1) \to \mathrm{Aff}(E^2/\Nu_2)$ is induced 
by conjugating by the affinity $(\ov\phi)_\star.$

It is now clear that the pair $(\Gamma_1, \Gamma_2)$ is enantiomorphic 
because in the orientation-preserving case, there is no element of $\mathrm{SL}(2,\integers)$ 
that conjugates  $AC$ to $(AC)^{-1}$, and in the orientation-reversing case, 
there is no element of  $\mathrm{GL}(2,\integers)$ of determinant $-1$ that conjugates $AC$ to $AC$. 
The enantiomorphic IT number pairs (144, 145), (169, 170), and (171, 172) are enantiomorphic by a similar argument. 

The second enantiomorphic pair $(\Gamma_1, \Gamma_2)$ has IT numbers 91 and 95. 
The co-Seifert fibration of $E^3/\Gamma_1$ is described in Table \ref{T17}. 
The action of the dihedral structure group of order 8 on $\circ \times \mathrm{O}$ is given by 
$($h-ref., ref.$)$, $($4-rot., 4-rot.$)$. 
The action of the structure group for $E^3/\Gamma_2$ is given by 
$($d-ref., ref.$)$, $($4-rot.$^{-1}$, 4-rot.$)$, 
since the classifying pair is  $\{$h-ref., d-ref.$\}$. 
Recall that h-ref.\,$ = A_\star$ and d-ref.\,$ = C_\star$. 

Let $\Nu_i$ be the unique 2-dimensional, complete, normal subgroup of $\Gamma_i$ for $i = 1, 2$. 
Then $\Gamma_i/\Nu_i$ is an infinite dihedral group for each $i = 1, 2$. 
By considering the image of a pair of Coxeter generators of $\Gamma_1/\Nu_1$, the same argument as before  
shows that the pair $(\Gamma_1, \Gamma_2)$ is enantiomorphic 
because in the orientation-preserving case, there is no element of $\mathrm{SL}(2,\integers)$ that conjugates the pair $(A, C)$ to the pair $(C, A)$, and in the orientation-reversing case,  
there is no element of  $\mathrm{GL}(2,\integers)$ of determinant $-1$ that conjugates the pair $(A, C)$ to itself.  
The enantiomorphic IT number pairs (92, 96), (151, 153), (152, 154), (178, 179), and (180, 181) are enantiomorphic 
by a similar argument.

\end{document}